\documentclass{amsart}
\usepackage{amssymb,amsmath,hyperref}

\newtheorem{theorem}{Theorem}[section]

\newtheorem{corollary}[theorem]{Corollary}

\newtheorem{lemma}[theorem]{Lemma}
\newtheorem{proposition}[theorem]{Proposition}

\numberwithin{equation}{section}

\newcommand{\A}{\mathbf{A}}
\newcommand{\Af}{\mathbf{A}_{\fin}}

\newcommand{\bs}{\backslash}

\newcommand{\comment}[1]{}
\newcommand{\C}{\mathbf{C}}

\newcommand{\ds}{\displaystyle}
\newcommand{\e}{\varepsilon}

\newcommand{\f}{f}

\newcommand{\fin}{\operatorname{fin}}

\newcommand{\g}{\gamma}

\newcommand{\GL}{\operatorname{GL}}

\newcommand{\mat}[4]{\begin{pmatrix} {#1} & {#2} \\ {#3} & {#4}
  \end{pmatrix}}

\newcommand{\meas}{\operatorname{meas}}
\renewcommand{\mod}{\text{ mod }}
\newcommand{\n}{n}

\newcommand{\ord}{\operatorname{ord}}

\newcommand{\ol}{\overline}
\newcommand{\olG}{\overline{G}}

\newcommand{\Q}{\mathbf{Q}}

\newcommand{\R}{\mathbf{R}}
\renewcommand{\Re}{\operatorname{Re}}

\newcommand{\sg}[1]{\left<{#1}\right>}
\newcommand{\sgn}{\operatorname{sgn}}
\newcommand{\smat}[4]{\bigl(\begin{smallmatrix}{#1}&{#2}\\{#3}&{#4}\end{smallmatrix}\bigr )}

\newcommand{\SL}{\operatorname{SL}}
\newcommand{\SO}{\operatorname{SO}}
\newcommand{\Span}{\operatorname{Span}}
\newcommand{\Supp}{\operatorname{Supp}}

\newcommand{\w}{\omega}

\newcommand{\Z}{\mathbf{Z}}
\newcommand{\Zhat}{\widehat{\Z}}

\newcommand{\mtxs}[4]{\left( \begin{smallmatrix} {#1} & {#2} \\ {#3} & {#4} \end{smallmatrix} \right)}

\begin{document}
\title{Averages of twisted $L$-functions}
\author{Julia Jackson}
\author{Andrew Knightly}
\address{Department of Mathematics \& Statistics\\University of Maine
\\Neville Hall\\ Orono, ME  04469-5752, USA }

\begin{abstract} We use a relative trace formula on $\GL(2)$ to compute a
  sum of twisted modular $L$-functions anywhere in the critical strip,
  weighted by a Fourier coefficient and a
  Hecke eigenvalue.  When the weight $k$ or level $N$ is sufficiently large, the
  sum is nonzero.  Specializing to the central point, we show in some cases
  that the resulting bound for the average is as good as that
  predicted by the Lindel\"of hypothesis in the $k$ and $N$ aspects.
\end{abstract}
\maketitle
\noindent \today
\thispagestyle{empty}

\section{Introduction} \label{intro}

  In many situations, the central $L$-values of modular forms 
  encode information about related algebraic objects.  For example,
  the non-existence of solutions to certain Diophantine equations can 
  hinge on the existence of cusp forms with non-vanishing
  central twisted $L$-value (see \cite{El1}, \cite{BEN}).
  Techniques from analytic number theory can then be used
  to estimate averages of $L$-values and thereby deduce the existence of
  such cusp forms.  The standard method, introduced by Duke \cite{Du},
  uses the Petersson trace formula together with Weil's bound for Kloosterman
  sums.  In the present paper, we use a different trace formula
  to compute the average of twisted $L$-functions directly at any point
  in the critical strip.  The resulting asymptotic formula has a much better
  error term (as a function of the level) and follows immediately
  without any use of regularization, approximate functional equations, 
  or deep results about Kloosterman sums.

Before stating the main result, we fix the following notation.
Let $S_k(N,\psi)$ be the space of cusp forms $h$ on $\Gamma_0(N)=\{\smat abcd\in \SL_2(\Z)|\,
  c\in N\Z\}$ satisfying
\[h(\frac{az+b}{cz+d})=\psi(d)(cz+d)^kh(z)\]
for all $z$ in the complex upper half-plane $\mathbf{H}$ and all $\smat abcd\in\Gamma_0(N)$,
  where $\psi$ is a Dirichlet character modulo $N$.
We normalize the Petersson inner product on $S_k(N,\psi)$ by
\begin{equation}\label{petnorm}
\|h\|^2=\frac1{\nu(N)}\iint_{\Gamma_0(N)\bs \mathbf H}|h(z)|^2y^k\frac{dx\,dy}{y^2},
\end{equation}
where
\[\nu(N)=[\SL_2(\Z):\Gamma_0(N)].\]

Given $h\in S_k(N,\psi)$, write $h(z)=\sum_{n>0}a_n(h)q^n$ for $q=e^{2\pi i z}$.
Fix an integer $D$ with $(D,N)=1$, and let $\chi$ be a primitive Dirichlet
  character modulo $D$.
The $\chi$-twisted $L$-function of $h$ is given for $\Re(s)>\tfrac k2+1$
   by the Dirichlet series
\[L(s,h,\chi)=\sum_{n>0}\frac{\chi(n)a_n(h)}{n^s}.\]
The completed $L$-function
\[\Lambda(s,h,\chi)=(2\pi)^{-s}\Gamma(s)L(s,h,\chi)\]
has an analytic continuation to the complex plane and 
  satisfies a functional equation relating $s$ to $k-s$, so the central point
  is $s=k/2$.  Taking $\chi$ trivial and $D=1$ gives the usual 
  $L$-function $\Lambda(s,h)$.
  When $N=1$, the functional equation takes the form
\begin{equation}\label{fe}
\Lambda(s,h,\chi)= \frac{i^k}{D^{2s-k}} \frac{\tau(\chi)^2}D \Lambda(k-s,h,\ol\chi),
\end{equation}
where
  \begin{equation}\label{tau}
\tau(\chi)=\sum_{m\in(\Z/D\Z)^*}\chi(m)e^{2\pi i m/D}
\end{equation}
is the Gauss sum attached to $\chi$.

Let $n$ be an integer prime to $N$, and let $T_n$ be the $n$-th Hecke operator, given by
\[T_n h(z) = n^{k-1}\sum_{ad=n,\atop{a>0}}\sum_{b=0}^{d-1}\psi(a)d^{-k}h(\frac{az+b}d).\]
Let $\mathcal{F}$ be an orthogonal basis for $S_k(N,\psi)$ consisting of eigenfunctions
  of $T_n$.  We denote the Hecke eigenvalue by $T_nh=\lambda_n(h)h$, and recall that
  \[a_n(h)=a_1(h)\lambda_n(h).\]
  Our main result is the following.

\begin{theorem}\label{main}
  With notation as above, assume $k>2$, and let $r,n\in \Z^+$ with $(rn,D)=1$.  
  Then for all $s=\sigma+i\tau$ in the strip $1<\sigma<k-1$,
\begin{align}\label{sum}
\frac1{\nu(N)}&
\sum_{h\in\mathcal{F}}
 \frac{\lambda_n(h)\,\ol{a_r(h)}\,\Lambda(s,h,\chi)}{\|h\|^2}\\
 \notag =&\frac{2^{k-1}(2\pi rn)^{k-s-1}}{(k-2)!}\Gamma(s)
  \sum_{d|(n,r)} d^{2s-k+1}\psi(\tfrac nd)\chi(\tfrac{rn}{d^2})
\\
\notag &+ \delta_{N,1}\frac{2^{k-1}(2\pi rn)^{s-1}}{(k-2)!}\Gamma(k-s)\frac{i^k}{D^{2s-k}}
\frac{\tau(\chi)^2}D \sum_{d|(r,n)}d^{k-2s+1} \ol{\chi(\tfrac{rn}{d^2})}\\
\notag &+E,
\end{align}
where $\delta_{N,1}\in\{0,1\}$ is nonzero iff $N=1$,  and
the error term $E$ is an infinite series
  involving confluent hypergeometric functions  (cf. Proposition \ref{deltat})
  satisfying
\begin{equation}\label{Ebound}
|E|\le 2\gcd(r,n)\frac{(4\pi rn)^{k-1}D^{k-\sigma-\frac12}\varphi(D)B(\sigma,k-\sigma)}
   {N^\sigma (k-2)!}
  \cosh(\tfrac{\pi\tau}2)\zeta(k-\sigma)\zeta(\sigma).
\end{equation}
  Here
  $B(x,y)=\tfrac{\Gamma(x)\Gamma(y)}{\Gamma(x+y)}$
  is Euler's Beta function and $\varphi(D)$ is Euler's $\varphi$-function.
\end{theorem}

Theorem \ref{main} extends first moment estimates of many authors.
  For prime level $N$, Duke estimated the average at the central point
  in the case $k=2$ and $r=n=1$, \cite{Du}.
  Akbary extended his result to allow for arbitrary weight $k$
  and summing over newforms (\cite{Ak}), and Kamiya treated
  arbitrary level and weight (with oldforms present)
  and $s$ any point on the critical line, \cite{Ka}.
  These references all make use of Petersson's formula,
  and obtain an error on the order of $O(N^{-k/4+\e})$, whereas \eqref{Ebound}
  is $O(N^{-k/2})$ for $s$ on the critical line.
  Ellenberg has shown how to refine Duke's method to improve the error bound
   to $O(N^{-k/2+\e})$ (\cite{El2}).

 In the case of twisting by a quadratic Dirichlet character when $N=1$,
 a different method was offered by Kohnen and Sengupta.  They
  gave an asyptotic for the average in the weight aspect
  using Waldspurger's formula relating the 
  central twisted $L$-values to certain Fourier coefficients 
  of half-integral weight modular forms, \cite{KS}.

  Here, we prove Theorem \ref{main} 
  by direct computation of a $\GL(2)$ relative trace formula involving integration
  over $N\times T$, where $N$ is unipotent and $T$ is a torus.
  This method was introduced in \cite{petrr}, in which the untwisted case was
  treated.  The incorporation of twisting 
  is achieved with an adelic twisting operator which we define in \S\ref{twist}.
  We have not made any attempt to go further and address the question of
  how many forms give a non-vanishing $L$-value.
  However, we note that by using estimates for mollified first and second moments,
  Iwaniec and Sarnak
  have shown that a positive proportion of cusp forms (in fact 50\% in 
  certain families) have nonvanishing quadratic-twisted central $L$-value,
   \cite{IS1}.
  Results of this nature have been used to bound the ranks of 
  Jacobians of modular curves (cf. \cite{IS}).

  Several authors have investigated first moments of Rankin-Selberg $L$-functions,
  i.e. where $\chi$ in \eqref{sum} is replaced with a fixed
  cusp form $h$ on $\GL(2)$.  In the case where $h$ is dihedral, one can
  do this via a relative trace formula on $T\times T$ (\cite{RR}, \cite{FW}),
  or by appealing to the Gross-Zagier formula (\cite{MR}). 
  Averages for more general $h$ have been studied recently by Nelson, Holowinsky and Templier 
  (level aspect, \cite{N}, \cite{HT}) and by Li and Masri (weight aspect, \cite{LM}).
  In several of the above references, a ``hybrid" subconvexity result is obtained,
  valid for forms whose level is in some range depending on the level of
  the fixed form $h$.
  Unfortunately, an analogous hybrid subconvexity result (in $N$ and $D$)
  is not possible in the present paper because of the poor control 
  of the $D$-aspect in \eqref{Ebound}.

An immediate application of Theorem \ref{main} is the nonvanishing of 
  $L$-functions:
 
\begin{corollary}\label{cor}
Suppose $N>1$ and $\gcd(r,n)=1$.
    Then for any $s$ in the critical strip $\tfrac{k-1}2<\Re(s)<\tfrac{k+1}2$,
  the sum \eqref{sum} is nonzero as long as $N+k$ is sufficiently large.
\end{corollary}
\noindent (See \S\ref{Asymp}, where the required size of $N+k$ as a function of $D$
  and $s$ may be ascertained.)
This can be interpreted as a GRH-on-average
  for the twisted $L$-functions, though
  no distinction is made between points on and off the critical line.
  When $N=1$, we cannot prove nonvanishing
  on the critical line $\Re(s)=\tfrac k2$, since the first two terms have the same
  magnitude there. Indeed, they cancel out at $s=\tfrac k2$ if $\chi$ is quadratic and
  conditions on $k,D$ conspire in \eqref{fe} to force the $L$-functions to vanish.
  However, by arguments given in \cite{petrr} one can show that when 
  $N=1$ and $\Re(s)\neq \tfrac k2$, the sum \eqref{sum} is nonzero if $k$ is sufficiently large.
  
  According to the generalized Lindel\"of hypothesis, for a newform $h$ we have
\begin{equation}\label{LH}
L(\tfrac k2,h,\chi)\ll (D^2Nk)^\e.
\end{equation}
  Let $\mathcal{F}_k(N)^{\text{new}}$ be any orthogonal basis for the span 
  $S_k(N)^{\text{new}}$ of the newforms with trivial character.
   Using the fact (\cite{Serre}, p. 86) 
  that $\dim S_k(N)^{\text{new}}\sim \frac{k-1}{12}\nu(N)^{\text{new}}$
  where $N^{1-\e}\ll\nu(N)^{\text{new}}\le N$, \eqref{LH} implies
  the following ``averaged" Lindel\"of hypo\-thesis:
\begin{equation}\label{LHave}
\sum_{h\in \mathcal{F}_k(N)^{\text{new}}}L(\tfrac k2,h,\chi)\ll D^\e(Nk)^{1+\e}.
\end{equation}
(This bound can only be expected to be an accurate prediction of the magnitude of the sum
  when the $L$-values are nonnegative.)

One can use Theorem \ref{main} to prove certain instances of \eqref{LHave} 
  unconditionally, although
  from \eqref{Ebound} it is clear that we cannot achieve adequate bounds
  in the $D$-aspect.  
  The idea is to set $n=r=1$ in \eqref{sum} and use well-known bounds for
  the Petersson norm, along with positivity of the $L$-values when $\chi$ is real.
  If oldforms are present, the method apparently grinds to a halt because
  $\frac{\ol{a_1(h)}\Lambda(\frac k2,h,\chi)}{\|h\|^2}$ may be negative.  Indeed,
  even in the simplest case where $N=p$ is prime, if $h$ is a newform of level $1$,
  and $h_p$ is a nonzero basis element (unique up to scaling) orthogonal to $h$ 
  in $\Span\{h(z),h(pz)\}$, then using \cite{ILS} (2.45)-(2.46)
   it is not hard to show that
\begin{equation}\label{hp}
\ol{a_1(h_p)}\Lambda(\tfrac k2,h_p,\chi)= \mu\left(\frac{\lambda_p(h)^2}{(p+1)^2}
 -\frac{\lambda_p(h)\chi(p)}{p^{1/2}(p+1)}\right)\Lambda(\tfrac k2,h,\chi)
\end{equation}
  for some constant $\mu>0$ depending on the choice of $h_p$.\footnote{By
  \cite{ILS} (2.45),
  $h_p=w\left(-\frac{\lambda_p(h)}{p+1}h(z)+p^{\frac{k-1}2}h(pz)\right)$
 for some nonzero $w\in\C$; by (2.46),
  $a_1(h_p)=-w\frac{\lambda_h(p)}{p+1}$; by an easy manipulation,
  $L(s,h(pz),\chi)=\frac{\chi(p)}{p^s}L(s,h,\chi)$, so $L(\tfrac{k}2,h_p,\chi)
  =w\left(-\frac{\lambda_p(h)}{p+1}
  +\frac{\chi(p)}{p^{1/2}}\right)L(\frac k2,h,\chi)$. Then \eqref{hp} holds with 
 $\mu=|w|^2$.}
  The above can clearly be negative, for example if
  $\chi(p)=1$ and the real eigenvalue $\lambda_p(h)$ is close 
  to $1$.
Therefore we have to content ourselves here with cases 
  in which oldforms are not present.  We highlight
  two such cases, one in the $k$-aspect and one in the $N$ aspect, though the
  proof applies more generally.

\begin{corollary} Let $\mathcal{F}_k(1)$ be an orthogonal basis for $S_k(1)$
  consisting of newforms, normalized with first Fourier coefficient equal to $1$.
  Then for any real primitive Dirichlet character $\chi$,
\begin{equation}\label{LHk}
\sum_{h\in\mathcal{F}_k(1)}L(\tfrac k2,h,\chi) \ll_{\e,D} k^{1+\e}.
\end{equation}
Let $4\le k_0\le 14$ be an even integer not equal to $12$, let $N$ be a prime not dividing 
 the conductor of $\chi$,
  and let $\mathcal{F}_{k_0}(N)^{\text{new}}$ be an orthogonal basis for
  $S_{k_0}(N)^{\text{new}}= S_{k_0}(N)$ consisting of normalized newforms.  Then
\begin{equation}\label{LHN}
\sum_{h\in\mathcal{F}_{k_0}(N)^{\text{new}}}L(\tfrac {k_0}2,h,\chi) \ll_{\e,D} N^{1+\e}.
\end{equation}
\end{corollary}

\noindent{\em Remarks:} The estimate \eqref{LHk} was first proven by
  Kohnen and Sengupta by different means (\cite{KS}).
  The case of trivial $\chi$ was proven earlier by Sengupta by essentially 
  the same method we use here (\cite{Se}).  The analog of \eqref{LHN} for 
  the second moment was
  established by Fomenko in case of trivial $\chi$ (\cite{Fo}).
  (By Cauchy-Schwarz, the estimate \eqref{LHN} is a consequence of its second
  moment analog.)

\begin{proof}
  In Theorem \ref{main}, 
  suppose that the central character $\w'$ is trivial and that $\chi$ is real.
  We assume that there are no oldforms, so $\mathcal{F}$ can be chosen to 
  consist of newforms $h$, normalized with $a_1(h)=1$.
  By the hypotheses on $\w'$ and $\chi$, we have $L(\frac k2,h,\chi)\ge 0$
  for all $h\in \mathcal{F}$ (\cite{Gu}).  Furthermore, we have the bound
\[\frac{(4\pi)^{k-1}}{(k-2)!}\nu(N)\|h\|^2\ll (kN)^{1+\e}\]
  for all newforms $h\in\mathcal{F}$ (see (2.29) of \cite{IM}).
  Therefore due to the nonnegativity of the $L$-values, we have
\begin{align*}
\sum_{h\in\mathcal{F}}L(\tfrac k2,h,\chi)\ll &\frac{(kN)^{1+\e}(k-2)!}{(4\pi)^{k-1}\nu(N)}
  \sum_{h\in\mathcal{F}}\frac{L(\frac k2,h,\chi)}{\|h\|^2}\\
  &=\frac{(kN)^{1+\e}(k-2)!}{2^{k-1}(2\pi)^{k/2-1}\Gamma(\frac k2)\nu(N)}\sum_{h\in\mathcal{F}}
  \frac{\Lambda(\frac k2,h,\chi)}{\|h\|^2}.
\end{align*}
Applying the theorem with $n=r=1$, we immediately obtain
\[\sum_{h\in\mathcal{F}}L(\tfrac k2,h,\chi)\ll (kN)^{1+\e}
  \left(1+\delta_{N,1}\frac{i^k\tau(\chi)^2}{D}
  + \frac{(k-2)!}{2^{k-1}(2\pi)^{k/2-1}\Gamma(\frac k2)}E\right).\]
It is clear from \eqref{Ebound} that the third term in the parentheses
  tends to $0$ as $N\to \infty$.  Using Stirling's approximation, it is
  not hard to show that the same is true as $k\to\infty$ (see \S \ref{Asymp} for
  details), and the corollary follows.
\end{proof}

\vskip .2cm
\noindent{\bf Acknowledgements.}  The first author was supported by
  a Chase Distinguished Research Assistantship from the 
  University of Maine Graduate School.  The third
  section of this paper is extracted from her Master's thesis.
  Both authors were supported by NSF grant DMS 0902145.
  We would like to thank the referee for carefully reading the manuscript 
  and offering insightful suggestions for improved exposition.

\section{Notation and preliminaries}

 Let $\A$, $\Af$ be the adeles and finite adeles of $\Q$, respectively. 
Fix a positive integer $N$.  For $x\in \A^*$, we let $x_N$
denote the idele whose $p$-th component is $x_p$ for all $p|N$ and $1$ for all
  $p\nmid N$.  For any integer $d$, we also write $d_p=\ord_p(d)$ (the $p$-adic
  valuation of $d$).  It should be
  clear from the context which meaning we take when a subscript $p$ appears.

Let $\psi$ be a Dirichlet character modulo $N$, extended to $\Z$ by
  $\psi(d)=0$ if $(d,N)>1$.  We let
$\psi^*$ denote its adelic counterpart (a Hecke character), defined via strong
  approximation $\A^*=\Q^*(\R^+\times \Zhat^*)$ by the pullback
\begin{equation}\label{pb}
\psi^*: \A^*\longrightarrow \Zhat^*\longrightarrow (\Z/N\Z)^*\longrightarrow\C^*,
\end{equation}
where the first arrows are the canonical projections, and the last arrow is $\psi$.
We drop the * from the notation for the local constituents.  Thus 
  $\psi_p: \Q_p^*\rightarrow \C^*$ is
  given by restricting $\psi^*$ to the embedded image of
  $\Q_p^*$ in $\A^*$.  
Note that if $d$ is an integer prime to $N$, then
\begin{equation}\label{psid}
\psi(d)=\prod_{p|N}\psi_p(d)=\psi^*(d_N).
\end{equation}
Later we will consider a character $\chi$ of modulus $D$, and all of the 
  above notation will apply equally with $D$ in place of $N$.

We let $\theta:\A\longrightarrow\C^*$ denote the standard character of $\A$,
  given locally by
\[\theta_p(x)=\begin{cases}e^{-2\pi i x}&\text{if }p=\infty \quad (x\in \R)\\
e^{2\pi i r_p(x)}&\text{if }p<\infty \quad (x\in \Q_p),\end{cases}\]
where $r_p(x)\in \Q$ is the $p$-principal part of $x$, a number with $p$-power 
  denominator characterized up to $\Z$
  by $x\in r_p(x)+\Z_p$.  The global character $\theta=\prod_{p\le \infty}\theta_p$
  is then trivial on $\Q$, and for finite $p$,
  $\theta_p$ is trivial precisely on $\Z_p$.  For $r\in \Q$, we define
\begin{equation}\label{thetar}
\theta_r(x)=\theta(-rx)=\ol{\theta(rx)}.
\end{equation}
  Every character of $\Q\bs \A$ is of the form $\theta_r$ for some $r\in \Q$.

Let $G$ denote the algebraic group $\GL_2$, with center $Z$, and let $\olG$ denote
  $G/Z$.
  The group $G(\Af)$ has the following sequences of open compact subgroups of $K=G(\Zhat)$:
\[K_0(N)=\{\smat abcd\in K|\, c\in N\Zhat\}\]
\[K_1(N)=\{\smat abcd\in K_0(N)|\, d\equiv 1\mod N\Zhat\}.\]
Because $\det K_0(N)=\det K_1(N)=\Zhat^*$, strong approximation holds for both of
  these, and in particular,
\begin{equation}\label{sa}
G(\A)=G(\Q)(G(\R)^+\times K_1(N)).
\end{equation}

Let $L^2(\psi^*)=L^2(\olG(\Q)\bs\olG(\A),\psi^*)$ be the space of measurable $\C$-valued
  functions $\phi$ on $G(\A)$ satisfying $\phi(z\g g)=\psi^*(z)\phi(g)$ for all 
  $z\in Z(\A), \g\in G(\Q), g\in G(\A)$, and which are square integrable over
  $\olG(\Q)\bs\olG(\A)$.  Let $L^2_0(\psi^*)$ denote the subspace of cuspidal functions.

We now normalize Haar measure on each group of interest.  Everything is the same
  as in \S 7 of \cite{KL}, where more detail is given.  On $\R$ we take Lebesgue
  measure $dx$, and we take $\tfrac{dy}{|y|}$ on $\R^*$.  We normalize
  the additive measure $dx$ on $\Q_p$ by taking $\meas(\Z_p)=1$, and likewise $d^*y$ on
  $\Q^*_p$ is normalized by $\meas(\Z_p^*)=1$.  These choices determine
  Haar measures on $\A$ and $\A^*$ in the usual way, with the property that
  $\meas(\A/\Q)=1$.  We give the compact
  abelian group $K_\infty=\SO(2)$ the measure $dk$ of total length $1$,
  and use the above measures
  to define measures on $N(\R)=\{\smat 1x01\}\cong \R$ and
  $M(\R)=\{\smat{y}{}{}z\}\cong \R^*\times \R^*$.  These choices determine 
  a Haar measure on $G(\R)$ by the Iwasawa decomposition: $dg=d(mnk)= dm\,dn\,dk$.
  In the same way, our fixed measures on $\Q_p$, $\Q_p^*$ determine measures
   on $N(\Q_p)$ and $M(\Q_p)$ respectively.  We take the
  unique measure on $G(\Q_p)$ for which the open compact subgroup
  $K_p=G(\Z_p)$ has measure $1$.  Let $Z$ denote the center of $G$ and set $\olG=G/Z$.
  We take $\meas(\ol{K_p})=1$ in $\olG(\Q_p)$.
  On $\olG(\R)$ we take the measure $d\ol m \, dn\,dk$, where $d\ol m$
  is the measure $\tfrac{dy}{|y|}$ on $\ol{M}(\R)\cong \{\smat y{}{}1\}\cong \R^*$.
  These local measures determine a
  Haar measure on $\olG(\A)$ for which $\meas(\olG(\Q)\bs\olG(\A))=\pi/3$.

Having fixed the measure, we note the following.
\begin{lemma}Let $D>0$ and let $\chi$ be a Dirichlet character modulo $D$ (not necessarily
  primitive), with Gauss sum $\tau(\chi)$ as in \eqref{tau}.  Let
  $\chi^*$ be the adelic realization of $\chi$ as in \eqref{pb}.
Then for any integer $n$ prime to $D$,
\begin{equation}\label{gauss}
\int_{\Zhat^*}\chi^*(u)\theta_{\fin}(\tfrac{nu}D)d^*u 
  = \frac{\ol{\chi(n)}}{\varphi(D)}\tau(\chi),
\end{equation}
where $\varphi$ is Euler's $\varphi$-function and $\theta_{\fin}=\prod_{p<\infty}\theta_p$.
\end{lemma}
\begin{proof}
  The integrand in \eqref{gauss} is invariant under the subgroup
\[U_D =(1+D\Zhat)\cap \Zhat^*=\prod_{p|D}(1+D\Z_p)\prod_{p\nmid D}\Z_p^*.\]
  Note that $\Zhat^*/U_D\cong(\Z/D\Z)^*$, so $\meas(U_D)=\varphi(D)^{-1}$.
  Therefore
\[
\int_{\Zhat^*}\chi^*(u)\theta_{\fin}(\tfrac{nu}D)d^*u = \frac1{\varphi(D)}
  \sum_{m\in(\Z/D\Z)^*} \chi^*(m_D)\theta_{\fin}(\tfrac{nm}D)
\]
\[=\frac1{\varphi(D)}\sum_{m\mod D} \chi(m)e^{2\pi i nm/D}
  =\ol{\chi(n)}\frac{\tau(\chi)}{\varphi(D)}.\qedhere\]
\end{proof}

We let
\[G(\R)^+=\{g\in G(\R)|\, \det g>0\}.\]
For $g=\smat abcd\in G(\R)^+$, we set
\[j(g,z)=\det(g)^{-1/2}(cz+d).\]
Recall that $j(g_1g_2,z)=j(g_1,g_2z)j(g_2,z)$.
The group action of $G(\R)^+$ on the complex upper half-plane $\mathbf{H}$
  by linear fractional transformations extends to a right action on the space of functions
$h:\mathbf H\longrightarrow\C$ via the weight $k$ slash operator
\[h|_g(z)=j(g,z)^{-k}h(g(z)) \qquad (g\in G(\R)^+, z\in \mathbf{H}).\]
Fix a Dirichlet character $\psi$ of modulus $N$, a positive integer $k$ satisfying
\begin{equation}\label{k}\psi(-1)=(-1)^k,
\end{equation}
  and let $S_k(N,\psi)$
   denote the space of cusp forms of level $N$, weight $k$, and
  character ${\psi}$.  Thus $h\in S_k(N,\psi)$ satisfies
\begin{equation}\label{Skdef}
h|_{\smat abcd} ={\psi(d)}\, h
\end{equation}
for all $\smat abcd\in \Gamma_0(N)$ and $z\in \mathbf H$.

The adelization of $h$ is the function $\phi_h\in L^2_0(\ol{\psi^*})$ defined
  using strong approximation \eqref{sa} by
\begin{equation}\label{phihdef}
\phi_h(\g(g_\infty\times g_{\fin}))= j(g_\infty,i)^{-k}h(g_\infty(i))
\end{equation}
for $\g\in G(\Q)$, $g_\infty\in G(\R)^+$, and $g_{\fin}\in K_1(N)$.
The modularity of $h$ makes $\phi_h$ well-defined, and one checks readily
  that the central character is indeed $\ol{\psi^*}$ (see e.g. the proof of Proposition
  4.5 of \cite{ftf};  the complex conjugate is needed here because we have not included
  it in \eqref{Skdef}.)
With the choice of Haar measure on $\olG(\A)$ given above, and the normalization
  \eqref{petnorm},
   the map $h\mapsto \phi_h$ is an isometry, i.e. $\|h\|=\|\phi_h\|$ (cf. (12.20) of
  \cite{KL}).

We recall the meaning of the following ``period integrals".
\begin{lemma}\label{basic}
For $r\in\Q$, and $h\in S_k(N,\psi)$,
\[\int_{\Q\bs\A}\phi_h(\smat{1}{x}{}{1})\,\ol{\theta_r(x)}dx=
  \left\{\begin{array}{cl}e^{-2\pi r}a_r(h)&\text{if }r\in \Z^+,\\
  0 &\text{otherwise}\end{array}\right.\]
and 
\[\int_{\Q^*\bs\A^*}\phi_h(\smat{y}{}{}1)\,|y|^{s-k/2}d^*y=\Lambda(s,h).\]
\end{lemma}
\begin{proof} See e.g. \cite{KL} Corollary 12.4 and \cite{petrr} Lemma 3.1
  respectively.
\end{proof}

\section{The twisting operator}\label{twist}

From now on, we assume that $\chi$ is a primitive Dirichlet character modulo $D$,
  with $(D,N)=1$.
Recall that $L(s,h,\chi)=L(s,h_\chi)$, where $h_\chi\in S_k(D^2N,\chi^2\psi)$ is
  given by
\[h_\chi(z) = \sum_{r=1}^\infty \chi(r) a_r(h) e^{2\pi i rz},\]
or equivalently,
\begin{equation}\label{hchi}
h_\chi =
  \frac1{\tau(\ol\chi)}\sum_{m\mod D}\ol{\chi(m)}h|_{\smat1{m/D}01}
\end{equation}
(see e.g. \cite{bump}, p. 59).
 Likewise, it follows from the definitions that for $g_\infty\in G(\R)^+$,
\begin{equation}\label{phihchi}
\phi_{h_\chi}(g_\infty\times 1_{\fin})=\frac1{\tau(\ol\chi)}\sum_{m\mod D}\ol{\chi(m)}
  \phi_h(\smat1{m/D}01g_\infty\times 1_{\fin}).
\end{equation}
  Because $\chi$ is assumed to be primitive, we have $|\tau(\ol\chi)|=\sqrt{D}$.

We now define a test function $f^\chi:G(\Af)\longrightarrow\C$ which 
  essentially realizes the twisting map $h\mapsto h_\chi$ adelically (see also \cite{RR}).
  It will be supported on the disjoint union
\begin{equation}\label{fchisupp}
\Supp(f^\chi)=\bigcup_{m\mod D,\atop{(m,D)=1}}\mat 1{-m/D}01 Z(\Af)K_1(N),
\end{equation}
where the rational matrix is embedded diagonally in $G(\Af)$.
The value of $f^\chi$ on the coset indexed by $m$ is defined to be
\begin{equation}\label{fchidef}
f^\chi(\smat1{-m/D}01 zk)=\frac{\nu(N) \ol{\chi(m)}{\psi^*(z)}}{\tau(\ol\chi)}.
\end{equation}
Here as before, 
\begin{equation}\label{nu}
\nu(N)=[K:K_0(N)]=[\ol{K}:\ol{K_1(N)}]=\meas(\ol{K_1(N)})^{-1}.
\end{equation}

\begin{lemma}\label{Jlem}
Consider the open compact subgroup
\[J=\{\smat abcd\in K_1(D^2N)|\,a\equiv 1\mod D\Zhat\}.\]
The function $f^\chi$ is right $K_1(N)$-invariant and left $J$-invariant.
\end{lemma}
\begin{proof} The first claim is obvious from the definition of $f^\chi$.
  The second claim follows from the fact that
\[\mat ab{cD^2N}d \mat 1{-\frac m D}01 = \mat 1{-\frac m D}01
  \mat{a+mcDN}{b+(d-a)\tfrac mD-m^2cN} {cD^2N}{d-mcDN},\]
noting that if $a\equiv 1\mod D$ and $d\equiv 1\mod D^2N$,
  the matrix on the right belongs to $K_1(N)$.
\end{proof}

Note that because $f^\chi$ has compact support modulo the center, and
  transforms under $Z(\Af)$ by ${\psi^*}$, it defines an operator on $L^2(\ol{\psi^*})$ by
\begin{equation}\label{Rf}
R(f^\chi)\phi(x)=\int_{\olG(\Af)}f^\chi(g)\phi(xg)dg.
\end{equation}
This operator is closely related to the twisting function $h\mapsto h_\chi$, as
  we now show.

\begin{proposition}
Let $h\in S_k(N,\psi)$ and
let $\chi^*$ be the Hecke character attached to $\chi$ as in \eqref{pb}.
Then for all $x\in G(\Af)$,
\begin{equation}\label{Rfchi}
R(f^\chi)\phi_h(x) =\chi^*(a_D)\phi_{h_\chi}(x),
\end{equation}
where $a$ is determined from $x$ using strong approximation by writing
\[x=x_\Q(x_\infty\times \smat abcd)\]
for $x_\Q\in G(\Q)$, $x_\infty\in G(\R)^+$ and $\smat abcd\in K_1(D^2N)$,
  and $a_D$ is the finite idele with local components $(a_D)_p=a_p$ (resp. $1$)
  if $p|D$ (resp. $p\nmid D$).
\end{proposition}
\noindent{\em Remark:} If $a\in \Zhat^*$, then $\chi^*(a_D)=\chi^*(a)$.

\begin{proof}
Clearly $R(f^\chi)\phi_h$ inherits the left $G(\Q)$-invariance from $\phi_h$.
Hence we can assume that $x_\Q=1$.
It is an easy consequence of the above lemma that $R(f^\chi)\phi$ is right $J$-invariant.
  Therefore we may modify $x_{\fin}$ on the right by an appropriate element of $J$
  to reduce to the case where $x_{\fin}=\smat {a_D}{}{}1$.
  Hence it suffices to prove that
\[R(f^\chi)\phi_h(x_\infty\times \smat {a_D}{}{}1)
  =\chi^*(a_D)\phi_{h_\chi}(x_\infty\times 1_{\fin}).\]

Let $\alpha_m = \smat{1}{m/D}{0}{1}$.  
From the preceding definitions and the right $K_1(N)$-invariance of $\phi_h$, 
  for any $x\in G(\A)$ we have
\[	R(f^\chi) \phi_h(x)
  = \sum_{m \bmod D} \int_{\alpha_m^{-1} \overline{K_1(N)}} f^\chi(g) \phi_h(xg) dg \]
\[
  = \sum_{m \bmod D} \phi_h(x \alpha_m^{-1}) \frac{\nu(N)\ol{\chi(m)}}{\tau(\ol\chi)}
   \int_{\alpha_m^{-1} \overline{K_1(N)}} dg\]
\[	 = \frac1{\tau(\ol\chi)} \sum_{m \bmod D} \overline{\chi(m)} \phi_h(x \alpha_m^{-1}).
\]
 Taking $x=x_\infty\times \smat {a_D}{}{}1$ we have
\[
	\phi_h(x \alpha_m^{-1}) 
   =\phi_h(x_\infty\times \smat{a_D}{}{}1\smat1{-\frac mD}01)
   =\phi_h(x_\infty\times \smat1{-\frac{a_Dm}D}01\smat{a_D}{}{}1)\]
\[   =\phi_h(x_\infty\times \smat1{-\frac{a_Dm}D}01)\]
by the right $K_1(N)$-invariance of $\phi_h$.  
Fix an integer $z$ relatively prime to $D$ such that $z\equiv a_D\mod D\Zhat$,
and multiply through on the left by $\alpha_{zm}$.
 By the left $G(\Q)$-invariance this has no effect, so the above is
\[=\phi_h(\smat1{\frac{zm}D}01x_\infty\times \smat1{\frac{(z-a_D)m}D}01)
=\phi_h(\alpha_{zm}x_\infty\times 1_{\fin}),\]
again by right $K_1(N)$-invariance.
Therefore
\[
R(f^\chi)\phi_h(x_\infty\times \smat{a_D}{}{}1)
  = \frac1{\tau(\ol\chi)}\sum_{m \bmod D} 
  \overline{\chi(m)} \phi_h(\alpha_{zm} x_\infty \times 1_\mathrm{fin})\]
\[ = \frac{\chi(z)}{\tau(\ol\chi)}\sum_{m \bmod D} 
  \overline{\chi(m)} \phi_h(\alpha_{m} x_\infty \times 1_\mathrm{fin})
  =\chi(z)\phi_{h_\chi}(x_\infty\times 1_{\fin})\]
by \eqref{phihchi}.
This gives the desired result, since $\chi(z)=\chi^*(a_D)$ by \eqref{pb} and \eqref{psid}.
\end{proof} 

It will be useful to define local components for $f^\chi$.
For any prime $p$, we have defined $\chi_p$ to be the character of $\Q_p^*$ attached
  to the Hecke character $\chi^*$.  When $p|D$, we have
\[\chi_p:\Z_p^*\longrightarrow (\Z_p/D\Z_p)^*\cong (\Z/p^{D_p}\Z)^*\longrightarrow\C^*.\]
Still assuming $p|D$, a local version of \eqref{gauss} is the following:
\begin{equation}\label{gaussp}
\int_{\Z_p^*}\chi_p(u)\theta_p(\tfrac{nu}D)d^*u=\frac{\ol{\chi_p(n)}}{\varphi(p^{D_p})}
  \tau(\chi)_p,
\end{equation}
where
\begin{equation}\label{taup}
\tau(\chi)_p=\chi_p(\tfrac{D}{p^{D_p}})\tau(\chi_p)
\end{equation}
for $\ds\tau(\chi_p)=\sum_{m\in(\Z/p^{D_p}\Z)^*}\chi_p(m)e^{2\pi im/p^{D_p}}$ 
the Gauss sum of the character $\chi_p$.
Then by \eqref{gaussp} and \eqref{gauss}, we have
\begin{equation}\label{taulocal}
   \tau(\chi)=\prod_{p|D}\tau(\chi)_p.
\end{equation}

  Given a prime $p$, we define a local function $f^\chi_p:G(\Q_p)\longrightarrow \C$ as follows.
If $p|D$, we take
\begin{equation}\label{fpsupp}
\Supp(\f^\chi_p)= \bigcup_{m\mod D\Z_p\atop{p\nmid m}}\smat{1}{-m/D}01Z_pK_p
\end{equation}
(a disjoint union), and define $f^\chi_p=\sum_m f^\chi_{p,m}$, where $f^\chi_{p,m}$
  is supported on the coset indexed by $m$ in \eqref{fpsupp}, and is given by
\begin{equation}\label{fchipm}
f^\chi_{p,m}(\smat 1{-m/D}01zk)=\frac{\ol{\chi_p(m)}\psi_p(z)}{\tau(\ol\chi)_p}
\end{equation}
for $\tau(\ol\chi)_p$ as in \eqref{taup}.
The value is independent of the choice of representative for $m\in (\Z_p/p^{D_p}\Z_p)^*$ 
  since $\chi_p$ has conductor $p^{D_p}$.
If $p\nmid D$, then $f^\chi_p$ is supported on $Z_pK_1(N)_p$, and we define it by
\[f_p^\chi(zk)= \nu_p(N)\psi_p(z),\]
where $\nu_p(N)=[K_p:K_0(N)_p]=\nu(p^{N_p})$.
It is easily verified using \eqref{psid} (applied to $\chi^*$) and
\eqref{taulocal} that $f^\chi=\prod_p f_p^\chi$.

\section{The Hecke operator}

We refer to \S13 of \cite{KL} for a more detailed account of the adelic
  Hecke operator defined here.
Fix a positive integer $n$ with $(n,DN)=1$.  
Define
\[M(n,N)=\{g=\smat abcd\in M_2(\Zhat)|\, \det g\in n\Zhat^*, c\in N\Zhat\}.\]
Define a function $f^n:G(\Af)\longrightarrow\C$ with
\[\Supp(f^n)= Z(\Af)M(n,N) = Z(\Q^+)M(n,N)\]
 by 
\[f^n(z_\Q m)=\nu(N)\psi^*(d_N)\quad(z_\Q\in Z(\Q^+), m=\smat abcd\in M(n,N)).\]
One shows easily that $f^n$ is bi-$K_1(N)$-invariant.

For any prime $p$, let
\[M(n,N)_p=\{g=\smat abcd\in M_2(\Z_p)|\, \det g\in n\Z_p^*, c\in N\Z_p\}.\]
Notice that if $p\nmid n$ then $M(n,N)_p=K_0(N)_p$.
Define a function $f^n_p:G(\Q_p)\rightarrow\C$, supported on $Z(\Q_p)M(n,N)_p$,
  by
\begin{equation}\label{fnp}
f^n_p(zm)=\nu_p(N)\psi_p(z)\psi_p((d_N)_p)
\quad(z\in Z(\Q_p), m=\smat abcd\in M(n,N)_p).
\end{equation}
Then it is straightforward to check that $f^n(g)=\prod_pf^n_p(g_p)$ for
  $g\in G(\Af)$.

\begin{proposition} For $h\in S_k(N,\psi)$,
\begin{equation}\label{Rfn}
R(f^n)\phi_h = n^{1-k/2}\phi_{T_n h}.
\end{equation}
\end{proposition}
\begin{proof}
See \cite{KL}, Proposition 13.6.
\end{proof}

\section{The global test function}

We take $f_\infty(g)={d_k}\ol{\sg{\pi_k(g)v_0,v_0}}$, where
  $\pi_k$ is the weight $k$ discrete series representation of $\GL_2(\R)$ with
  formal degree $d_k=\frac{k-1}{4\pi}$, central character $\smat x{}{}x\mapsto\sgn(x)^k$,
  and lowest weight unit vector $v_0$.
  Explicitly,
\[
f_\infty(\mat abcd)=
\frac{(k-1)}{4\pi}\frac{\det(g)^{k/2}(2i)^k}{(-b+c+(a+d)i)^k}
\]
if $ad-bc>0$, and it vanishes otherwise (see \cite{KL}, Theorem 14.5).
This function is self-adjoint, meaning
\begin{equation}\label{finfsa}
f_\infty(g)=\ol{f_\infty(g^{-1})}.\end{equation}
It is integrable if and only if $k>2$ (\cite{KL}, Prop. 14.3).

Given two functions $f_1,f_2\in L^1(\psi^*)$, we define their convolution
\[f_1*f_2(x)=\int_{\olG(\A)}f_1(g)f_2(g^{-1}x)dg
  =\int_{\olG(\A)}f_1(xg^{-1})f_2(g)dg.\]
It is straightforward to show that
\begin{equation}\label{*}
R(f_1*f_2)=R(f_1)\circ R(f_2)
\end{equation}
as operators on $L^2(\ol{\psi^*})$, where
\[R(f)\phi(x)=\int_{\olG(\A)}f(g)\phi(xg)dg.\]

Fix an integer $n>0$ relatively prime to $DN$, and set
\begin{equation}\label{fconv}
f=(f_\infty\times f^\chi)*(f_\infty\times f^n).
\end{equation}
Local components for $f$ can be defined as follows.
\begin{proposition}
With notation as in the previous two sections, define
\[f_p = \begin{cases} f^\chi_p=f^n_p &\text{if }p\nmid nD\\
  f^\chi_p &\text{if }p|D\\
  f^n_p & \text{if }p|n.\end{cases}\]
Then $f=f_\infty\prod_p f_p$.
\end{proposition}
\begin{proof}
Because $f_\infty\times f^\chi$ and $f_\infty\times f^n$ are both factorizable
  and identically $1$ on $K_p$ for a.e. $p$, the integral defining their
  convolution is factorizable, and hence
\[f=(f_\infty\times f^\chi)*(f_\infty\times f^n)=(f_\infty*f_\infty)\prod_p
  (f^\chi_p*f^n_p).\]
It follows directly from the orthogonality relations for discrete series that
  $f_\infty*f_\infty=f_\infty$.  Indeed,
\[f_\infty*f_\infty(x)=d_k^2\int_{\olG(\R)}\ol{\sg{\pi_k(g)v_0,v_0}\sg{\pi_k(g^{-1}x)v_0,v_0}}dg\]
\[ = d_k^2\int_{\olG(\R)}\sg{\pi_k(g)v_0,\pi_k(x)v_0}\ol{\sg{\pi_k(g)v_0,v_0}}dg
 = d_k^2\frac{\sg{v_0,v_0}\ol{\sg{\pi_k(x)v_0,v_0}}}{d_k}=f_\infty(x).\]
Likewise, simple direct computation shows that for finite primes $p$,
   $f^\chi_p*f^n_p=f_p$ as given.
\end{proof}
\noindent Globally, the support of $f$ is
\begin{equation}\label{fsup}
\Supp(f)=G(\R)^+\times \bigcup_{m\mod D\atop{(m,D)=1}}\mat{1}{-m/D}01Z(\Q^+)M(n,N).
\end{equation}
The union over $m$ is easily seen to be disjoint, using the fact that $(n,D)=1$.
  Accordingly, we can write
\begin{equation*}
f_{\fin}=\sum_{m\in (\Z/D\Z)^*} f_m,
\end{equation*}
where $f_m$ is supported on the coset indexed by $m$ in \eqref{fsup}, and
\begin{equation}\label{fm}
f_m(\smat 1{-m/D}01 zk)= \frac{\nu(N)\ol{\chi(m)}\psi^*(d_N)}{\tau(\ol\chi)}
\end{equation}
  for $z\in Z(\Q^+)$ and $k=\smat abcd\in M(n,N)$.

  Under the condition
  $k>2$ (which will be in force throughout) $f\in L^1(\psi^*)$, so
  the operator $R(f)$ on $L^2(\ol{\psi^*})$ is defined.

\begin{proposition}\label{Rfproj}
The operator $R(f)$ factors through the
  orthogonal projection of $L^2(\ol{\psi^*})$ onto the finite dimensional subspace
  $S_k(N,\psi)$ (embedded in $L^2(\ol{\psi^*})$ via \eqref{phihdef}).
\end{proposition}
\begin{proof}
The operator $R(f_\infty\times f^n)$ factors through the orthogonal projection
  onto $S_k(N,\psi)$ (\cite{KL}, Corollary 13.13).  Therefore by \eqref{*} and
  \eqref{fconv},
  $R(f)$ has the same property.
\end{proof}
\noindent{\em Remark:} The image of $R(f)$ is {\em not} contained in any classical 
  space of cusp forms $S_k(\Gamma_1(M))$.
  Indeed, it is immediate from \eqref{Rfchi} that the image of $R(f)$ is
  not 
  left $K_1(M)$-invariant for any $M$, since a congruence condition on 
  $a$ is needed.\\

The operator $R(f)$ is an integral operator given by the continuous kernel
\begin{equation}\label{kerdef}
K(g_1,g_2)=\sum_{h\in \mathcal{F}}\frac{R(f)\phi_h(g_1)\ol{\phi_h(g_2)}}{\|h\|^2}
  =\sum_{\g\in \olG(\Q)}f(g_1^{-1}\g g_2).
\end{equation}
Here the spectral sum is taken over any orthogonal basis $\mathcal{F}$
  for $S_k(N,\psi)$, as a consequence of Proposition \ref{Rfproj}, and both
  sums are absolutely convergent.

\section{Spectral side}
Fix a positive integer $r$ relatively prime to $D$.
We prove Theorem \ref{main} by computing the following integral
\begin{equation}\label{ktf}
\int_{\Q^*\bs\A^*}\int_{\Q\bs\A}
  K(\mat y001,\mat{1}{x}{0}{1})\,\theta_r(x)\ol{\chi^*(y)}|y|^{s-k/2}dx\, d^*y
\end{equation}
using the two expressions for the kernel \eqref{kerdef}.
We note that the double integral is absolutely convergent for all $s\in\C$ (see below).

On the spectral side, \eqref{ktf} becomes
\begin{equation}\label{spec0}
\sum_{h\in\mathcal{F}}\frac1{\|h\|^2}\int_{\Q^*\bs\A^*}R(f)\phi_h(\smat y{}{}{1})
  \ol{\chi^*(y)}|y|^{s-k/2}d^*y\cdot
  \int_{\Q\bs\A}\ol{\phi_h(\smat 1x01)}\theta_r(x)dx.
\end{equation}
Choose $\mathcal{F}$ in \eqref{kerdef} to consist of eigenvectors of $T_\n$.  
  Then for $h \in \mathcal{F}$, we write $T_n h=\lambda_n(h) h$, so that
  by \eqref{Rfchi} and \eqref{Rfn},
\[R(f)\phi_h(\smat y{}{}1)=n^{1-k/2}\lambda_n(h)\chi^*(y)\phi_{h_\chi}(\smat y{}{}1)\]
 for all $y\in \R^+\times \Zhat^*\cong \Q^*\bs\A^*$.
   Consequently the factor $\ol{\chi^*(y)}$ in \eqref{spec0} is cancelled out, and
   \eqref{spec0} becomes
\begin{equation}\label{spec}
\sum_{h\in\mathcal{F}}\frac{n^{1-k/2}\lambda_n(h)}{\|h\|^2}
\int_{\Q^*\bs\A^*}\phi_{h_\chi}(\smat{y}{}{}1)|y|^{s-k/2}d^*y
  \int_{\Q\bs\A}\ol{\phi_h(\smat1x{}1)}{\theta_r(x)}dx
 \end{equation}
\begin{equation}\label{spec2}=
\frac{n^{1-k/2}}{e^{2\pi r}}\sum_{h\in\mathcal{F}}
 \frac{\lambda_n(h)\ol{a_r(h)}}{\|h\|^2}\Lambda(s,h,\chi)
\end{equation}
by Lemma \ref{basic}.
The two integrals in \eqref{spec} are absolutely convergent for all $s$, so we
  have the following.

\begin{proposition} The double integral \eqref{ktf} is absolutely convergent for all $s\in\C$.
\end{proposition}

\section{Geometric side}

For the moment, let $H(\A)=\ol{M}(\A)\times N(\A)\cong \A^*\times \A$.
Inserting the geometric expression  $K(g_1,g_2)=\sum_\g f(g_1^{-1}\g g_2)$ 
  into \eqref{ktf}, we get
\[
\int_{H(\Q)\bs H(\A)}
  \sum_{\g\in\olG(\Q)}f(\smat {y^{-1}}001 \g \smat{1}{x}{0}{1})\,\theta_r(x)
  \ol{\chi^*(y)}|y|^{s-k/2}dx\, d^*y\]
\[=\int_{H(\Q)\bs H(\A)}
  \sum_\delta\sum_{\g\in[\delta]}f(\smat {y^{-1}}001 \g \smat{1}{x}{0}{1})\,\theta_r(x)
  \ol{\chi^*(y)}|y|^{s-k/2}dx\, d^*y,\]
where $\delta$ ranges over a set of representatives for the 
  $H(\Q)$-orbits in $\olG(\Q)$ relative to the action
$(m,n)\cdot \g=m^{-1}\g n$,
and $[\delta]=\{m^{-1}\delta n|\, (m,n)\in H_\delta(\Q)\bs H(\Q)\}$ is the
  orbit.  It is not hard to check that in fact for all $\delta$, the stabilizer
  $H_\delta(\Q)=\{1\}$.  Therefore,
  \eqref{ktf} is formally equal to 
\begin{equation}\label{geom}
\sum_{\delta}\int_{\A^*}\int_{\A}f(\smat {y^{-1}}{}{}1\delta\smat{1}{x}{}{1})
  \theta_r(x) \ol{\chi^*(y)}|y|^{s-k/2}dx\,d^*y,
\end{equation}
where $\delta$ runs through a set of representatives for 
  $\ol{M}(\Q)\bs\olG(\Q)/N(\Q)$.
By the Bruhat decomposition
\[G(\Q)=M(\Q)N(\Q)\cup M(\Q)N(\Q)\smat{}{-1}{1}{}N(\Q),\]
a set of representatives $\delta$ is given by
\begin{equation}\label{deltaset}
 \{1\} \cup \{\smat 0{-1}{1}0\}\cup \{\smat{t}{-1}{1}0|\, t\in\Q^*\}.
\end{equation}
The equality between \eqref{ktf} and \eqref{geom} is valid on the
  strip $1<\Re(s)<k-1$.  This is a consequence of the following.

\begin{proposition}\label{conv} Suppose $1<\Re(s)<k-1$.  Then
\[\sum_{\delta}\int_{\A^*}\int_{\A}\left|f(\smat{y^{-1}}{}{}1\delta\smat{1}{x}{}{1})
  \theta_r(x)\ol{\chi^*(y)}|y|^{s-k/2}\right| dx\,d^*y <\infty.\]
\end{proposition}
\begin{proof}
 This is proven in just the same way as the analogous result in \cite{petrr}, Proposition 3.3.
 We outline the steps.  Because $f_{\fin}$ is bounded and compactly supported
  as a function of $y,x$, the argument hinges on bounding the infinite part
\[I_\delta^{abs}(f)_\infty = \int_0^\infty\int_{-\infty}^\infty|f_\infty(
  \smat{y^{-1}}{}{}1\delta\smat 1x01)|dx\,y^{\sigma-k/2-1}dy,\]
where $\sigma=\Re(s)$.
However, by \eqref{finfsa} the above is
\[= \int_0^\infty\int_{-\infty}^\infty|f_\infty(
  \smat 1{-x}01\delta^{-1}\smat{y}{}{}1)| dx\,y^{\sigma-k/2-1}dy.\]
Noting that the set of $\delta$ in \eqref{deltaset} is exactly the set of inverses
  of the $\delta$'s in \cite{petrr}, the above coincides with 
  $I_{\delta^{-1}}^{abs}(f)_\infty$
  considered in \S 3.3 there.   Hence those results give
\[I_\delta^{abs}(f) < \infty \quad\text{for } \begin{cases} \delta = 1\text{ and }
  0<\sigma <k-1\\\delta =\smat{}{-1}1{}\text{ and } 1<\sigma <k\\
  \delta = \smat t{-1}10\text{ and } 0<\sigma<k.\end{cases}\]
Furthermore, by Proposition 3.3 of \cite{petrr},
  for $\delta_t = \smat t{-1}10$ we have
\begin{equation}\label{tinfbound}
I_{\delta_t}^{abs}(f)_\infty \ll|t|^{\sigma-k} \quad\text{if }  0<\sigma<k.
\end{equation}
Thus, to complete the proof it remains to show that for $1<\sigma<k-1$,
\[\sum_{t\in\Q^*}I_{\delta_t}^{abs}(f)<\infty.\]
We will prove in Proposition \ref{deltat} below that the finite part 
  $I_{\delta_t}^{abs}(f)_{\fin}$
  vanishes unless $t=\tfrac{Nb}{nD}$ for some $b\in\Z-\{0\}$.
  We will show in \eqref{dtfin} that for such $t$,
\[|I_{\delta_t}^{abs}(f)_{\fin}|\le \frac{n^{\sigma-k/2}\nu(N)\varphi(D)\gcd(r,n)}
  {N^{2\sigma-k}D^{1/2}}\sum_{d|b} d^{k-2\sigma}.\]
Together with \eqref{tinfbound}, the fact that $\#\{d: d|b\}\ll |b|^\e$
  for $\e>0$, and using $d^{k-2\sigma}\le 1$ when $\sigma>k/2$,
  this gives the global estimate
\[\sum_{t\in\Q^*}I_{\delta_t}^{abs}(f)\ll_{N,D,n,\e}\begin{cases}
  \sum_{b\in \Z-\{0\}}|b|^{-\sigma+\e}&\text{if }\sigma\le k/2\\
  \sum_{b\in \Z-\{0\}}|b|^{\sigma-k+\e}&\text{if }\sigma> k/2.\end{cases}\]
This is finite when $1<\sigma<k-1$ and $\e$ is sufficiently small.
\end{proof}

Let $I_\delta(s)$ denote the double integral attached to $\delta$ in \eqref{geom}.
For
  $1<\Re(s)<k-1$ and each $\delta$ in \eqref{deltaset},
we need to compute $I_\delta(s)$.  It factorizes as
\[I_\delta(s)=I_\delta(s)_\infty I_\delta(s)_{\fin}=I_\delta(s)_\infty\prod_{p} I_\delta(s)_p,\]
where 
\[I_\delta(s)_\infty=\int_{\R^*}\int_{\R}f_\infty(\smat{y^{-1}}{}{}1\delta\smat1x01)
  \ol{\theta_\infty(rx)}\ol{\chi_\infty(y)}|y|^{s-k/2}dx\,\tfrac{dy}{|y|}\]
and likewise
\[I_\delta(s)_p=\int_{\Q_p^*}\int_{\Q_p}f_p(\smat{y^{-1}}{}{}1\delta\smat1x01)\ol{\theta_p(rx)}
  \ol{\chi_p(y)}|y|_p^{s-k/2}dx\,d^*y.\]
From the definition of $f_\infty$, the integrand for $I_\delta(s)_\infty$
   vanishes unless $y>0$.  Because $\chi_\infty$ is trivial on $\R^+$, it has
  no effect on $I_\delta(s)_\infty$ and can be removed.
Using \eqref{finfsa}, we find that
\begin{equation}\label{Iinf}I_\delta(s)_\infty 
  =\ol{I_{\delta^{-1}}'(\ol s)_\infty},
\end{equation}
where
\[I_{\delta^{-1}}'(s)_\infty=\int_0^\infty\int_{-\infty}^\infty f_\infty(
  \smat1{-x}01\delta^{-1}\smat y{}{}1)\theta_\infty(rx)y^{s-k/2}dx\,\tfrac{dy}y\]
is the archimedean factor computed in \cite{petrr}.

For convenience, when computing the finite part $I_\delta(s)_{\fin}$
  we will replace $y$ by $y^{-1}$.
  (It is a property of unimodular (e.g. abelian) groups that this does not affect 
  the value of the integral.)  Thus
\[I_\delta(s)_{\fin}=\int_{\Af^*}\int_{\Af}f_{\fin}(\smat y{}{}1\delta\smat 1x01)
  \ol{\theta_{\fin}(rx)}dx\,\chi^*(y)|y|_{\fin}^{k/2-s}d^*y.\]
Looking at determinants, by \eqref{fsup} the integrand is nonzero 
  only if $y=\tfrac{nu}{\ell^2}$ for some
  $\ell\in \Q^+$ and $u\in\Zhat^*$.  Thus, since $\Q^+\cap \Zhat^*=\{1\}$, the above is
\[=\sum_{\ell\in\Q^+}\Bigl(\frac n{\ell^2}\Bigr)^{s-k/2}
  \int_{\Zhat^*}\int_{\Af}f_{\fin}(\smat{\ell}{}{}{\ell}^{-1}
  \smat {\frac{nu}{\ell}}{}{}\ell\delta\smat 1x01)
  \ol{\theta_{\fin}(rx)}dx\,\chi^*(\tfrac{nu}{\ell^2})d^*u.\]
Note that $\chi^*_{\fin}(\tfrac{n}{\ell^2})=\chi^*(\tfrac{n}{\ell^2})=1$ since $\tfrac n{\ell^2}
  \in \Q^+$.  Likewise, the scalar factor of $\ell_{\fin}^{-1}$ pulls 
  out of $f_{\fin}$  as ${\psi^*_{\fin}(\ell^{-1})}=1$.  Hence
\begin{equation}\label{Ifin}
I_\delta(s)_{\fin}=\sum_{\ell\in\Q^+}\Bigl(\frac n{\ell^2}\Bigr)^{s-k/2}
  \int_{\Zhat^*}\int_{\Af}f_{\fin}(\smat{\frac{nu}{\ell}}{}{}\ell\delta\smat 1x01)
  \ol{\theta_{\fin}(rx)}dx\,\chi^*(u)d^*u.
\end{equation}

\begin{proposition}\label{delta1} When $\delta = 1$, the integral
\[I_1(s)=\int_{\A^*}\int_\A f(\mat{y}{xy}{0}{1})\ol{\theta(rx)}dx
   \,{\chi^*(y)}|y|^{k/2-s}d^*y
\]
converges absolutely on $0<\Re(s)<k-1$, and for such $s$ it is equal to
\begin{equation}\label{I1}
\frac{n^{1-k/2}}{e^{2\pi r}} \frac{2^{k-1}(2\pi rn)^{k-s-1}}{(k-2)!}\Gamma(s)
\nu(N) \sum_{d|(n,r)} d^{2s-k+1}\psi(\tfrac nd)\chi(\tfrac{rn}{d^2}).
\end{equation}
\end{proposition}

\begin{proof}
  The absolute convergence was proven in Proposition \ref{conv}.
We factorize the integral as $I_1(s)=I_1(s)_\infty I_1(s)_{\fin}$.  By \eqref{Iinf}
  and the proof of Proposition 3.4 of \cite{petrr}, we have
\begin{equation}\label{I1inf}
I_1(s)_\infty= \frac{2^{k-1}(2\pi r)^{k-s-1}}{(k-2)!\,e^{2\pi r}}\Gamma(s).
\end{equation}
Now consider the finite part, which by \eqref{Ifin} is
\[I_1(s)_{\fin}=\sum_{\ell\in\Q^+}\Bigl(\frac n{\ell^2}\Bigr)^{s-k/2}
  \int_{\Zhat^*}\int_{\Af}f_{\fin}(\mat{\frac{nu}{\ell}}{\frac{xnu}\ell}0\ell)
  \ol{\theta_{\fin}(rx)}dx\,\chi^*(u)d^*u.\]
Replacing $x$ by $\tfrac{\ell x}{nu}$, the above is
\[=\sum_{\ell\in\Q^+} \frac{n^{s-k/2+1}}{\ell^{2s-k+1}}
  \sum_{m\in(\Z/D\Z)^*}\int_{\Zhat^*}\int_{\Af}f_{m}(\mat {\frac{nu}{\ell}}{x}0\ell)\,
  \ol{\theta_{\fin}(\tfrac{r\ell x}{nu})}dx\,\chi^*(u)d^*u\]
\[=\sum_{\ell\in\Q^+} \frac{n^{s-k/2+1}}{\ell^{2s-k+1}}
  \sum_{m}\int_{\Zhat^*}\int_{\Af}f_{m}(\smat1{-m/D}01
  \mtxs{\frac{nu}{\ell}}{x+\frac{\ell m}D}0\ell)\,
  \ol{\theta_{\fin}(\tfrac{r\ell x}{nu})}dx\,\chi^*(u)d^*u.\]
  By \eqref{fsup}, the integrand is nonzero if and only if 
  $\tfrac{nu}\ell,\ell,x+\tfrac{\ell m}D\in\Zhat$.  Replacing $x$ by $x-\tfrac{\ell m}D$,
  we obtain
\[I_1(s)_{\fin}=\frac{\nu(N)}{\tau(\ol{\chi})}
\sum_{\ell|n} \frac{n^{s-k/2+1}}{\ell^{2s-k+1}}\psi^*(\ell_N)
  \sum_{m}\ol{\chi(m)}\int_{\Zhat^*}\int_{\Zhat}
  \ol{\theta_{\fin}(\tfrac{r\ell(x-\frac{\ell m}D)}{nu})}dx\,\chi^*(u)d^*u\]
\[=\frac{\nu(N)}{\tau(\ol{\chi})}
\sum_{\ell|n} \frac{n^{s-k/2+1}}{\ell^{2s-k+1}}\psi(\ell)
  \sum_{m}\ol{\chi(m)} \int_{\Zhat^*} \theta_{\fin}(\tfrac{r\ell^2 m}{nDu})
\int_{\Zhat} \ol{\theta_{\fin}(\tfrac{r\ell x}{nu})}dx\,\chi^*(u)d^*u.\]
The integral over $\Zhat$ evaluates to $1$ if $\tfrac n\ell|r$, and $0$ otherwise.
  Setting $d=\tfrac n\ell$, the above is
\[=\frac{\nu(N)}{\tau(\ol{\chi})}
\sum_{d|(n,r)} \frac{n^{s-k/2+1}}{(\frac nd)^{2s-k+1}}\psi(\tfrac nd)
  \sum_{m}\ol{\chi(m)} \int_{\Zhat^*} \theta_{\fin}(\tfrac{rn m}{d^2Du})
\chi^*(u)d^*u\]
\[=\frac{\nu(N)n^{k/2-s}}{\tau(\ol{\chi})}
\sum_{d|(n,r)} d^{2s-k+1}\psi(\tfrac nd)
  \sum_{m}\ol{\chi(m)} \int_{\Zhat^*} \theta_{\fin}(\tfrac{rn m}{d^2D}u)
\ol{\chi^*(u)}d^*u\]
\[=\frac{\nu(N)n^{k/2-s}}{\tau(\ol{\chi})}
\sum_{d|(n,r)} d^{2s-k+1}\psi(\tfrac nd)
  \sum_{m}\ol{\chi(m)} \chi(\tfrac{rnm}{d^2})\frac{\tau(\ol\chi)}{\varphi(D)}\]
\[={\nu(N)n^{k/2-s}} \sum_{d|(n,r)} d^{2s-k+1}\psi(\tfrac nd)\chi(\tfrac{rn}{d^2}).\]
Passing to the third line, we applied \eqref{gauss}.  Multiplying by \eqref{I1inf}
  gives the result.
\end{proof}
Although we computed the orbital integral globally, it may be of interest to
  know the value of the local orbital integrals, for example if one wishes
  to compute the analogous trace formula over a number field, or use a test function
  which differs from ours at a finite number of places.  Letting
  \[I_1(s)_p=\int_{\Q_p^*}\int_{\Q_p}f_p(\smat{y}{xy}{0}1)\ol{\theta_p(rx)}
  \chi_p(y)|y|_p^{k/2-s}dx\,d^*y,\]
 by calculations very similar to the above, we find, for $r\in\Z_p$,
\begin{equation}\label{I1p}
I_1(s)_p=\begin{cases}\chi_p(r)&(p|D)\\\nu(p^{N_p})&(p|N)\\
\ds(p^{n_p})^{k/2-s}\sum_{d_p=0}^{\min(r_p,n_p)}(p^{d_p})^{2s-k+1}\psi_p(\tfrac
  {p^{d_p}}{p^{n_p}})\chi_p(\tfrac{p^{2d_p}}{p^{n_p}})&(p|n)\\
  1&(p\nmid nND).
\end{cases}
\end{equation}

\begin{proposition}\label{deltaw} When $\delta = \smat{}{-1}1{}$, the integral
\[I_\delta(s)=\int_{\A^*}\int_\A f(\mat{0}{-y}{1}{x})\ol{\theta(rx)}dx
   \,{\chi^*(y)}|y|^{k/2-s}d^*y
\]
converges absolutely on $1<\Re(s)<k$, and for such $s$ it vanishes unless $N=1$.
 When $N=1$ (so $k$ is even by \eqref{k}),
\begin{equation}\label{Iw}
I_\delta(s)=\frac{n^{1-k/2}}{e^{2\pi r}}
\frac{2^{k-1}(2\pi rn)^{s-1}}{(k-2)!}\Gamma(k-s)\frac{i^k}{D^{2s-k}}
\frac{\tau(\chi)^2}D
  \sum_{d|(r,n)}d^{k-2s+1} \ol{\chi(\tfrac{rn}{d^2})}.
\end{equation}
\end{proposition}
\noindent{\em Remark:} Comparing with the identity term \eqref{I1} when $N=1$, we see that 
  \[ \frac{i^k}{D^{2s-k}} \frac{\tau(\chi)^2}D I_1(k-s,\ol\chi)=I_\delta(s,\chi),\]
  mirroring the functional equation \eqref{fe} on the spectral side.

\begin{proof}
  The absolute convergence was proven in Proposition \ref{conv}.
By \eqref{Iinf} and the proof of Proposition 3.5 of \cite{petrr}, we have
\begin{equation}\label{Idinf}
I_\delta(s)_\infty = \frac{2^{k-1}(2\pi r)^{s-1}}{e^{2\pi r}(k-2)!\,i^k}\Gamma(k-s).
\end{equation}
Now consider the finite part \eqref{Ifin}:
\[I_\delta(s)_{\fin}=\sum_{\ell\in\Q^+}\Bigl(\frac n{\ell^2}\Bigr)^{s-k/2}
  \int_{\Zhat^*}\int_{\Af}f_{\fin}(\mat0{-\frac{nu}{\ell}}\ell{x\ell})
  \ol{\theta_{\fin}(rx)}dx\,\chi^*(u)d^*u.\]
Because the above matrix has determinant $nu\in n\Zhat^*$, we see from \eqref{fsup} that
  \begin{equation}\label{ellp}
\mat0{-\frac{nu_p}\ell}\ell{\ell x_p}\in M(n,N)_p \quad\text{for all }p\nmid D
\end{equation}
when $\ell x_p\in \Z_p$.
  Likewise, 
   we can assume that for some $m\in(\Z/D\Z)^*$,
\begin{equation}\label{ellD}
\mat1{\frac mD}01 \mat0{-\frac{nu}\ell}\ell{x\ell}=\mat{\frac{\ell m}D}{-\frac{nu}\ell+
  \frac{m\ell x}D}{\ell}{\ell x}\in M(n,N).
\end{equation}
The latter implies that $x\ell\in\Zhat$ and $\ell\in N\Z^+$. If $p|N$,
  then by considering the upper right entry of \eqref{ellp}, we have
  $\ord_p(n)\ge \ord_p(\ell)\ge \ord_p(N)>0$.  This contradicts $(n,N)=1$, and therefore
  we may assume that $N=1$.
   From the upper left entry of \eqref{ellD}, we
  see that $D|\ell$.  Write $\ell=Dd$ for $d\in\Z^+$.  Replacing $x$ by $\tfrac x\ell
  =\tfrac x{Dd}$, the measure is scaled by $|Dd|_{\fin}^{-1}=Dd$,
   so $I_\delta(s)_{\fin}$ is equal to
\[\sum_{d\in\Z^+}\tfrac {n^{s-k/2}}{(Dd)^{2s-k-1}}
  \sum_{m\in (\Z/D\Z)^*}\int_{\Zhat^*}\int_{\Zhat}f_m(\smat1{-\frac mD}01
  \mtxs{md}{\frac{-nu}{Dd}+\frac{mx}D}{Dd}{x})
  \ol{\theta_{\fin}(\tfrac{rx}{Dd})}dx\,\chi^*(u)d^*u.\]
From the upper right-hand entry of \eqref{ellD}, 
$-\tfrac{nu}{Dd}+\tfrac{mx}D\in\Zhat$.  Since $x\in\Zhat$, this means
\begin{equation}\label{mdx}
mdx\in md\Zhat\cap (nu+Dd\Zhat).
\end{equation}
Generally, it is not hard to show that for any $h,j,k\in\Zhat$,
\begin{equation}\label{coset}
h\Zhat\cap(j+k\Zhat)=\begin{cases} hc+\frac{hk}{\gcd(h,k)}\Zhat&\text{if }\gcd(h,k)|j\\
\emptyset&\text{if }\gcd(h,k)\nmid j,\end{cases}
\end{equation}
where $c\in\Z$ is any solution to $hc\equiv j\mod k\Zhat$.
Applying this to \eqref{mdx}, we have $\gcd(h,k)=\gcd(md,Dd)=d$, so
  the set in \eqref{mdx} is nonempty if and only if $d|n$.  Assuming this holds,
  the range of $x$ is determined by
\[x\in c_m+D\Zhat,\]
where $c_m$ is any positive integer satisfying
\begin{equation}\label{cm}
mdc_m\equiv nu\mod dD\Zhat.
\end{equation}
For such $x$, the value of $f_m$ is identically equal to 
  $\tfrac{\ol{\chi(m)}}{\tau(\ol\chi)}$ since $N=1$ (and so $\psi=1$).
Replace $x$ by $c_m+Dx$.  This changes the measure by a factor of $|D|_{\fin}=D^{-1}$, and
\[  I_\delta(s)_{\fin}=
\frac{n^{s-k/2}}{\tau(\ol\chi)D^{2s-k}}\sum_{d|n}d^{k-2s+1}
  \sum_{m}\ol{\chi(m)}\int_{\Zhat^*} \int_{\Zhat}
  \ol{\theta_{\fin}(\tfrac{r(c_m+Dx)}{Dd})}dx\,\chi^*(u)d^*u.\]
The inner integral $\int_{\Zhat}\ol{\theta_{\fin}(\tfrac{rx}d)}dx$ is equal to $1$ if
  $d|r$, and $0$ otherwise.  Therefore the above is
\[=\frac{n^{s-k/2}}{\tau(\ol\chi)D^{2s-k}}\sum_{d|(r,n)}d^{k-2s+1}
  \sum_{m}\ol{\chi(m)}\int_{\Zhat^*}
  \ol{\theta_{\fin}(\tfrac{rc_m}{Dd})}\chi^*(u)d^*u.\]
Since $d|(r,n)$, \eqref{cm} is equivalent to $c_m\equiv \ol m(\tfrac nd)u\mod D\Zhat$,
  where $m\ol m\equiv 1\mod D$.  Using this along with \eqref{gauss}, we find:
\[I_\delta(s)_{\fin}=\frac{n^{s-k/2}}{\tau(\ol\chi)D^{2s-k}}\sum_{d|(r,n)}d^{k-2s+1}
  \sum_{m}\ol{\chi(m)}\int_{\Zhat^*}
  {\theta_{\fin}\left(\frac{-(\frac rd)(\frac nd)\ol m u}{D}\right)}\chi^*(u)d^*u\]
\[=\frac{n^{s-k/2}}{\tau(\ol\chi)D^{2s-k}}\sum_{d|(r,n)}d^{k-2s+1}
  \sum_{m}\ol{\chi(m)}\chi(-1)\ol{\chi(\tfrac{rn}{d^2})}\chi(m)\frac{\tau(\chi)}{\varphi(D)}\]
\begin{equation}\label{Iwfin}
=\frac{n^{s-k/2}}{D^{2s-k}}\frac{\tau(\chi)}{\chi(-1)\tau(\ol{\chi})}
  \sum_{d|(r,n)}d^{k-2s+1} \ol{\chi(\tfrac{rn}{d^2})}
\end{equation}
Since $\chi$ is primitive, we have $\tau(\chi)\chi(-1)\tau(\ol\chi)=
  \tau(\chi)\ol{\tau(\chi)}=D$.  Therefore
$\frac{\tau(\chi)}{\chi(-1)\tau(\ol\chi)}$ $=\frac{\tau(\chi)^2}D$.  Using this and
  multiplying the above by \eqref{Idinf}, equation \eqref{Iw} follows.
\end{proof}

We state here the value of the local orbital integrals
\[I_\delta(s)_p=\int_{\Q_p^*}\int_{\Q_p}f_p(\smat0{-y}1x)\ol{\theta_p(rx)}
  \chi_p(y)|y|_p^{k/2-s}dx\,d^*y.\]
By local calculations similar to the above, we find, for $\delta=\smat0{-1}10$,
\begin{equation}\label{Iwp}
I_\delta(s)_p=\begin{cases}(p^{D_p})^{k-2s}\ol{\psi_p(D)}\,\ol{\chi_p(D^2)}
  \,\ol{\chi_p(-r)}\frac{\tau(\chi)_p}{\tau(\ol\chi)_p}&(p|D)\\
  0&(p|N)\\
  \ds(p^{n_p})^{s-k/2}\sum_{d_p=0}^{\min(r_p,n_p)}(p^{d_p})^{k-2s+1}\ol{\psi_p(p^{d_p})}
  \chi_p(\tfrac{p^{n_p}}{p^{2d_p}})&(p|n)\\
  1&(p\nmid nND).
\end{cases}
\end{equation}
Here, recall that $\tau(\chi)_p=\chi_p(\tfrac{D}{p^{D_p}})\tau(\chi_p)$ as in \eqref{taup}.
Using \eqref{taulocal}, it is straightforward to show that the product 
  of the above over all $p$ agrees with \eqref{Iwfin} when $N=1$.

\section{Computation of $I_{\delta_t}(s)$}

In this section we will prove the following.

\begin{proposition}\label{deltat} When $\delta_t = \smat{t}{-1}1{0}$, the integral
\[I_{\delta_t}(s)=\int_{\A^*}\int_\A f(\mat{yt}{yxt-y}{1}{x})\ol{\theta(rx)}dx
   \,{\chi^*(y)}|y|^{k/2-s}d^*y
\]
converges absolutely on the strip $0<\sigma <k$, where $\sigma=\Re(s)$.
    The integrand vanishes unless $t=\tfrac{Nb}{nD}$ for
  $b\in\Z-\{0\}$.  When $1<\sigma<k-1$, the sum $\sum_{t\in\Q^*}I_{\delta_t}(s)$
  is absolutely convergent, and $E:=\frac{e^{2\pi r}}{\nu(N)n^{1-k/2}}\sum I_{\delta_t}(s)$
  is equal to
\[ \hskip -.3cm \frac{(4\pi rn)^{k-1}\varphi(D)\psi(nD)e^{i\pi s/2}}{N^sD^{s-k}(k-2)!\,\tau(\ol\chi)}
\hskip -.4cm\sum_{a\neq 0,d>0\text{ sat.}\eqref{aD},\atop{\gcd(a,Nd^{(D)})|\gcd(r,n)}}
\hskip -.4cm\frac{a^{s-k}\gcd(a,Nd^{(D)})}{d^s\psi(a)e^{\frac{2\pi i r\ell }{ad_D}}}J_\chi(a,d)
{}_1f_1(s;k;\tfrac{-2\pi irnD}{Nad}),\]
where: $a^s=e^{-i\pi s}|a|^s$ if $a<0$, 
$d^{(D)}=\prod_{p\nmid D}p^{d_p}$ is the prime-to-$D$ part of $d$,
 and similarly for $d=d^{(D)}d_D$, $\ell$ is any integer satisfying
  $Nd^{(D)}\ell\equiv -nD\mod ad_D$, $J_\chi$ is a product of certain explicit local
  factors of absolute value $\le 1$ given in \eqref{Jchi}, 
\begin{equation}\label{1f1}
_1f_1(s,k;w)=\frac{\Gamma(s)\Gamma(k-s)}{\Gamma(k)}
  {}_1\mathrm{F}_1(s;k;w)=\int_0^1 e^{wx}x^{s-1}(1-x)^{k-s-1}dx
\end{equation}
 for $\Re(k)>\Re(s)>0$,
and writing $a_p,d_p$
  for the $p$-adic valuations of $a,d$, we have
\begin{equation}\label{aD}
p|D\implies \begin{cases} a_p=D_p&\text{if }d_p> D_p\\
  a_p\ge D_p&\text{if }d_p=D_p\\
  a_p=d_p&\text{if }0\le d_p<D_p.
\end{cases}
\end{equation}
\end{proposition}

\noindent{\em Remark:}
We give an expression for $I_{\delta_t}(s)$
  in \eqref{Idt} below.  As in \cite{petrr}, this can
  be used in principle to compute the sum over $t$ to any level of precision.\\

  The absolute convergence was proven in Proposition \ref{conv}.
To begin the computation, write $\delta =\delta_t$, and consider the finite part given by \eqref{Ifin}:
\[I_\delta(s)_{\fin}=\sum_{\ell\in\Q^+}\left(\frac n{\ell^2}\right)^{s-k/2}
  \int_{\Zhat^*}\int_{\Af}f_{\fin}(
\mat{\frac{nut}\ell}{\frac{nutx}\ell -\frac{nu}\ell}{\ell}{\ell x}) 
  \ol{\theta_{\fin}(rx)}dx\,\chi^*(u)d^*u.\]
We will show that this vanishes unless $t\in \tfrac{N}{nD}\Z$.  In anticipation of this,
  write $t=\tfrac{Nb}{nD}$, where (for now) $b\in\Q^*$.
Since the determinant of the matrix is $nu\in n\Zhat^*$, by \eqref{fsup}
  the integrand vanishes unless $\ell\in N\Z^+$ and $\ell x\in\Zhat$.
  Therefore we shall set $\ell =Nd$, and replace $x$ by $\ell^{-1}x=x/Nd$, so 
  $dx$ becomes $d(x/Nd)=|Nd|_{\fin}^{-1}dx=Nd\cdot dx$.
  The above then becomes
\begin{equation}\label{dsum}
\sum_{d\in \Z^+}\frac{n^{s-k/2}Nd}
{(Nd)^{2s-k}}
  \int_{\Zhat^*}\int_{\Zhat}
 f_{\fin}(\mat{\frac{ub}{dD}}{\frac{ubx}{Nd^2D}-\frac{nu}{Nd}}{Nd}{x})
  \ol{\theta_{\fin}(\tfrac{rx}{Nd})}dx\,\chi^*(u)d^*u.
\end{equation}
We will show that the integrand vanishes unless $b\in\Z$ and $d|b$.
We now work locally.  The $p$-th factor of the double integral is
\begin{equation}\label{fpt}
  \int_{\Z_p^*}\int_{\Z_p}
 f_{p}(\mat{\frac{ub}{dD}}{\frac{ubx}{Nd^2D}-\frac{nu}{Nd}}{Nd}{x})
  \ol{\theta_{p}(\tfrac{rx}{Nd})}dx\,\chi_p(u)d^*u.
\end{equation}

\subsection{Local computation at $p|D$}

  Suppose first that $p|D$.  Then $N$ is a unit, so replacing $u$ by $Nu$, 
  \eqref{fpt} becomes
\begin{equation}\label{fptD}  \int_{\Z_p^*}\int_{\Z_p}
 f_{p}(\mat{\frac{Nub}{dD}}{\frac{ubx}{d^2D}-\frac{nu}{d}}{Nd}{x})
  \ol{\theta_{p}(\tfrac{rx}{Nd})}dx\,\chi_p(Nu)d^*u.
\end{equation}
Recall that $f_p=\sum_{m\in (\Z_p/D\Z_p)^*} f_{p,m}$, where $f_{p,m}=f^\chi_{p,m}$ is
  given in \eqref{fchipm}.  
  Fix $m$ and consider 
\[ f_{p,m}(\mat{\frac{Nub}{dD}}{\frac{ubx}{d^2D}-\frac{nu}{d}}{Nd}{x})
=f_{p,m}(\smat{1}{-\frac mD}01 \mat{\frac{Nub}{dD}+\frac{Ndm}D}
  {\frac{ubx}{d^2D}-\frac{nu}{d}+\frac {mx}D}{Nd}{x}).\]
By \eqref{fpsupp} and the fact that the determinant of the rightmost 
  matrix is $nNu\in \Z_p^*$, this is nonzero if and only if
   the rightmost matrix belongs to $K_p$,  or equivalently:
\[\begin{array}{lll}
(\text{i})\quad \tfrac{ub}{dD} +\tfrac {dm}D\in\Z_p\\
(\text{ii})\quad \tfrac x{d}(\tfrac{ub}{dD}+\tfrac{dm}D)-\tfrac {nu}{d}\in\Z_p.
\end{array}\]
We assume henceforth that these conditions hold.
Notice that if $p\nmid d$, the first condition already implies the second.
On the other hand, since $x\in\Z_p$,
\[p|d, \text{(i)}, \text{(ii)}\implies \begin{array}{lll}
(\text{iii})\quad x\in\Z_p^*\\
(\text{iv})\quad (\tfrac{ub}{dD} +\tfrac{dm}D)\in\Z_p^*.
\end{array}\]
Letting $D_p=\ord_p(D)$, and similarly for $b_p,d_p$,
  we find by condition (i) (if $d_p=0$) and condition
  (iv) (if $d_p>0$) that
\begin{equation}\label{bD}
p|D\implies\begin{cases} b_p=d_p+D_p&\text{if }d_p> D_p\\
b_p\ge 2D_p &\text{if }d_p=D_p\\
b_p=2d_p&\text{if }0\le d_p <D_p.
\end{cases}
\end{equation}

Suppose first that $p\nmid d$, so that by \eqref{bD}, $d_p=b_p=0$.
Then condition (i) is equivalent to 
\[m\equiv \frac{-bu}{d^2}\mod D\Z_p,\]
$d,b,u,m$ being units.  So given $u$, there is exactly one $m$
  for which the above condition holds, and by \eqref{fchipm} the inner integral
  of \eqref{fptD} is equal to
\[  \int_{\Z_p}\tfrac{\ol{\chi_p(m)}}{\tau(\ol\chi)_p}\ol{\theta_p(\tfrac{rx}{Nd})}dx
  =\frac{\ol{\chi_p(\frac{-b}{d^2})}}
  {\tau(\ol{\chi}_p)}\ol{\chi_p(u)}\]
since $\theta_p(\tfrac{rx}{Nd})=1$.
Therefore the double integral \eqref{fpt} is equal to
\begin{equation}\label{dp0}
\chi_p(N)\frac{\chi_p(\frac{-d}{b/d})}{\tau(\ol\chi)_p}
  \int_{\Z_p^*}\ol{\chi_p(u)}\chi_p(u)d^*u
=\frac{\chi_p(\frac{-Nd}{b/d})}{\tau(\ol\chi)_p}.
\end{equation}

Now suppose $p|d$. In view of (ii) and (iv), we have
\[x\in \frac{Dnu}{u\frac {b}d+d{m}}+d\Z_p\subset \Z_p^*.\]
This is the only condition on $x$ required for the $f_{p,m}$-term to be nonzero.
Make the substitution $x= \frac{Dnu}{ub/d+dm}+d\cdot w$, so $dx=|d|_pdw=|Nd|_pdw$. 
The value of $f_{p,m}$ is $\tfrac{\ol{\chi_p(m)}}{\tau(\ol\chi)_p}$,
  so assuming (iv) holds, the inner integral in 
  \eqref{fptD} is equal to
\[|Nd|_p\sum_m
  \frac{\ol{\chi_p(m)}}{\tau(\ol\chi)_p}
 \ol{\theta_p(\tfrac{rDnu}{Nub+Nd^2m})}
  \int_{\Z_p} \ol{\theta_p(\tfrac{rw}N)}dw.\]
The latter integral has value $1$ since $r\in \Z^+$ and $N$ is a unit.
The variable $u$ ranges through the set 
 $U_{b,d,m}=(-\tfrac{d^2m}b+\tfrac{dD}b\Z_p^*)\cap \Z_p^*$
  determined by condition (iv) above.  By considering the possibilities for $d_p>0$
  listed in \eqref{bD}, we find easily that
\[U_{b,d,m} = \left\{
\begin{array}{cl} 
  \Z_p^*&\text{if }d_p>D_p\text{ (so }b_p=d_p+D_p)\\
-\frac{d^2m}b+\frac Dd\Z_p^*&
  \text{if } 0<d_p<D_p \text{ (so }b_p=2d_p)\\
   \Z_p^*&\text{if }d_p=D_p \text{ and }b_p>2D_p\\
\displaystyle\bigcup_{a\in (\Z/p\Z)^*,\atop {a\not\equiv -\frac{d^2m}{b}\mod p}}(a+p\Z_p)&
  \text{if } d_p=D_p\text{ and }b_p=2D_p.
\end{array}\right.\]
The double integral \eqref{fpt} is equal to
\[
\frac{|Nd|_p} {\tau(\ol\chi)_p}
\sum_{m\in(\Z_p/D\Z_p)^*}\ol{\chi_p(m)}
  \int_{U_{b,d,m}}\chi_p(Nu)\ol{\theta_p(\tfrac{rDnu}{Nub+Nd^2 m})}d^*u.
\]
Noting that $U_{b,d,m}=mU_{b,d,1}$, we can replace $u$ by $mu$ and integrate
  over $U_{b,d,1}$. 
  This has the effect of cancelling every $m$, so that the above is
\begin{equation}\label{IdtD}
\frac{|Nd|_p\varphi_p(D)} {\tau(\ol\chi)_p}
  \int_{U_{b,d,1}}\chi_p(Nu)\ol{\theta_p(\tfrac{rDnu}{Nub+Nd^2})}d^*u.
\end{equation}
We leave this as something that could be computed given $\chi_p$, if desired.
  For our purposes, it will be enough to bound the integral trivially by $1$.
(We do not think that a more careful treatment
  of the integral can yield enough power saving in $D$
  to enable the type of hybrid subconvexity bound mentioned in the introduction.)

\subsection{Local computation at $p\nmid D$}

Now we suppose $p\nmid D$.  In this case, $\chi_p$ is unramified, so \eqref{fpt}
is equal to
\begin{equation}\label{fpnN}
  \int_{\Z_p^*}\int_{\Z_p}
 f_{p}(\mat{\frac{ub}{dD}}{\frac{ubx}{Nd^2D}-\frac{nu}{Nd}}{Nd}{x})
  \ol{\theta_{p}(\tfrac{rx}{Nd})}dx\,d^*u.
\end{equation}
Since the support of $f_p$ is $Z(\Q_p)M(n,N)_p$ and the determinant of the above matrix
  is $nu\in n\Z_p^*$, the integrand is nonzero precisely when
\[\begin{array}{l}
(\text{i})\quad \tfrac{ub}{dD}\in\Z_p\\
(\text{ii})\quad \tfrac{x}{Nd}(\tfrac{ub}{dD})-\tfrac{nu}{Nd}\in\Z_p.
\end{array}
\]
Both conditions are in fact independent of $u$.
By (i), we see that $0\le d_p \le b_p$ since $p\nmid D$.
    Together with \eqref{bD}, this proves our assertion
  that $I_\delta(s)$ vanishes unless $b\in \Z$, and that 
  the sum in \eqref{dsum} can be taken just over $d|b$.  
Since $u,D\in\Z_p^*$, condition (ii) is equivalent to
\begin{equation}\label{bdx}
\tfrac bdx\in ({Dn}+{Nd}\Z_p)\cap \tfrac bd\Z_p.
\end{equation}
(If $p|N$, this is possible only if $d_p=b_p$.)
  Applying the local analog of \eqref{coset} to \eqref{bdx}, and then 
  dividing by $\tfrac bd$, we find that
\[x\in \begin{cases} c+\tfrac{Nd}{\gcd(b/d,Nd)}\Z_p&\text{if } \gcd_p(b/d,Nd)|Dn\\
\emptyset&\text{otherwise,}\end{cases}
\]
 where $\gcd_p$ denotes the $p$-part of the $\gcd$, and $c\in\Z$ is given by
\[\tfrac bdc\equiv Dn\mod Nd\Z_p.\]
We shall specify $c$ further as follows, so that the above holds simultaneously for
  all $p\nmid D$.
  Write $d=d^{(D)}d_D$, where $(D,d^{(D)})=1$ and $d_D=\prod_{p|D}p^{d_p}$.
    Then we take $c\in\Z$ so that:
\begin{equation}\label{c}
\begin{cases}\tfrac bdc\equiv Dn\mod Nd^{(D)}\Z\\
c\equiv 0\mod d_D\Z.\end{cases}
\end{equation}
 It is not hard to see that such $c$ exists under the hypothesis $\gcd_p(b/d,Nd)|Dn$
  for all $p\nmid D$.  Indeed, $\prod_{p\nmid D}\gcd_p(\tfrac bd,Nd)
  =\gcd(\tfrac bd,Nd^{(D)})|Dn$, which implies
  the existence of an integer $c$ satisfying the 
  first congruence.  If necessary we can 
  mulitply $c$ by $d_D\ol{d_D}\equiv 1\mod Nd^{(D)}$ to further ensure that $c\in d_D\Z$.

The first congruence in \eqref{c} implies that $b/d$ is relatively prime to $N$.
Therefore $\psi_p(x)=\psi_p(c)$.
Since $\meas(\Z_p^*)=1$ and everything is independent of $u\in \Z_p^*$,
   the double integral \eqref{fpnN} is equal to
\[
\nu_p(N)\psi_p((c_N)_p)\int_{c+\frac{Nd}{\gcd(b/d,Nd)}\Z_p}\ol{\theta_p(\tfrac{rx}{Nd})}dx,
\]
  where we used the formula \eqref{fnp} for $f_p$.  Now let
  $x=c+\tfrac{Nd}{\gcd(b/d,Nd)}w$.  Then the above is
\begin{equation}\label{IdtN}
=\nu_p(N)\psi_p((c_N)_p)\left|\tfrac{Nd}{\gcd(b/d,Nd)}\right|_p\ol{\theta_p(\tfrac{rc}{Nd})}
  \int_{\Z_p}\ol{\theta_p(\tfrac{rw}{\gcd(b/d,Nd)})}dw.
\end{equation}
The integral is nonzero (hence equal to $1$) if and
  only if $\gcd_p(b/d,Nd)|r$.

\subsection{The finite part}
Multiply the local factors \eqref{IdtD} (resp. \eqref{dp0}) and \eqref{IdtN}, together
  with the coefficient of the double integral in \eqref{dsum}.  We set
\begin{equation}\label{Jchi}
J_\chi(\tfrac bd,d)=\prod_{p|D}J_p(\tfrac bd,d),
\end{equation}
where $J_p(\tfrac bd,d)$ denotes the integral in \eqref{IdtD} 
   if $d_p>0$
  (resp. the quantity $\frac{\chi_p(\frac{-Nd}{b/d})}{\varphi_p(D)}$
 if $d_p=0$).
  We find:
\begin{align*} I_{\delta_t}(s)_{\fin}
  =  \hskip -.5cm \sum_{d|b \text{ satisfying } \eqref{bD},
\atop{ \gcd(\frac bd,Nd^{(D)})|(r,n)}}
  &\frac{n^{s-k/2}Nd} {(Nd)^{2s-k}}J_\chi(\tfrac bd,d)
  \prod_{p|D}\frac{|Nd|_p\varphi_p(D)}{\tau(\ol\chi)_p}\\
&\times \prod_{p\nmid D}
\nu_p(N)\psi_p((c_N)_p)\left|\tfrac{Nd}{\gcd(\frac bd,Nd)}\right|_p
  \ol{\theta_p(\tfrac{rc}{Nd})}.
\end{align*}
Here $d^{(D)}=\prod_{p\nmid D}p^{d_p}$ as before.
We can make a few simplifications.
First, 
\[\prod_{p\nmid D}\psi_p((c_N)_p)=\psi^*(c_N)=\psi(c)=\frac{\psi(nD)}{\psi(b/d)}\]
since $\tfrac bdc\equiv nD\mod N$ and $(\tfrac bd,N)=1$ by \eqref{bdx}.
By the second congruence of \eqref{c}, namely $d_D|c$, we have
   $\theta_p(\tfrac{rc}{Nd})=1$ for all $p|D$.  Hence
\[\prod_{p\nmid D} \ol{\theta_p(\tfrac{rc}{Nd})}=\ol{\theta_{\fin}(\tfrac{rc}{Nd})}
  =\theta_\infty(\tfrac{rc}{Nd})=e^{-\frac{2\pi i rc}{Nd}}.\]
Therefore
\begin{equation}\label{dtfin}
I_{\delta_t}(s)_{\fin} = \frac{n^{s-k/2}\varphi(D)\nu(N)}
  {N^{2s-k}\tau(\ol\chi)}
    \hskip -.5cm \sum_{d|b \text{ satisfying } \eqref{bD},
\atop{ \gcd(\frac bd,Nd^{(D)})|(r,n)}}
\hskip -.1cm \frac{\psi(nD)}{\psi(\frac bd)}
  \frac{\gcd(\tfrac bd,Nd^{(D)})}{d^{2s-k}
  e^{\frac{2\pi i rc}{Nd}}}J_\chi(\tfrac bd,d).
\end{equation}

\subsection{Archimedean integral and global expression}

Finally, we consider the archimedean integral
$I_{\delta_t}(s)_\infty$. 
By \eqref{Iinf} and the proof of Proposition 3.7 of \cite{petrr}, we have
\[I_{\delta_t}(s)_\infty=\ol{\frac{(4\pi r)^{k-1}t^{\ol s-k}}{(k-2)!\,e^{2\pi r}}
 e^{-i\pi \ol s/2}e^{-2\pi i r/t}{}_1f_1(\ol s;k;2\pi ir/t)},\]
where $t^s=e^{i\pi s}|t|^s$ if $t<0$.
    Therefore
\[I_{\delta_t}(s)_\infty=\frac{(4\pi r)^{k-1}t^{s-k}e^{i\pi s/2}e^{2\pi i r/t}}{(k-2)!\,e^{2\pi r}}
{{}_1f_1(s;k;-2\pi ir/t)},
\]
where now $t^s=e^{-i\pi s}|t|^s$ if $t<0$.
 By the discussion above, we can take $t=Nb/nD$, so 
\begin{equation}\label{Idtinf}
I_{\delta_t}(s)_\infty=
\frac{(4\pi r)^{k-1}N^{s-k}e^{i\pi s/2}}{(k-2)!\,e^{2\pi r}n^{s-k}D^{s-k}}
b^{s-k}e^{\frac{2\pi i rnD}{Nb}}{{}_1f_1(s;k;\tfrac{-2\pi irnD}{Nb})}.
\end{equation}
When we multiply by the finite part \eqref{dtfin}, the terms $e^{2\pi i rnD/Nb}$
  and $e^{-2\pi i rc/Nd}$ combine as follows.  By \eqref{c} we can
  write $(\tfrac bd)c=nD+\ell Nd^{(D)}$, where $\ell\in\Z$ is an integer satisfying
\[\ell Nd^{(D)} \equiv -nD\mod (\tfrac bd)d_D.\]
  Conversely, any $\ell$ satisfying the above determines an integer $c$ satisfying
  \eqref{c}.  Then
\[e^{\frac{2\pi i rnD}{Nb}}e^{-\frac{2\pi i rc}{Nd}}=e^{\frac{2\pi i r(nD-\frac bdc)}{Nb}}
  =e^{\frac{-2\pi i r\ell Nd^{(D)}}{Nb}}=e^{\frac{-2\pi i r \ell }{(b/d)d_D}}.\]
Multiplying \eqref{Idtinf} by the finite part \eqref{dtfin}, we find, for $t=\tfrac{Nb}{nD}$,
\begin{align}\label{Idt}
\notag I_{\delta_t}(s)=& 
\frac{(4\pi r)^{k-1}N^{s-k}e^{i\pi s/2}}{(k-2)!\,e^{2\pi r}n^{s-k}D^{s-k}}
  b^{s-k} {{}_1f_1(s;k;\tfrac{-2\pi irnD}{Nb})}\\
  &\times 
 \frac{n^{s-k/2}\varphi(D)\nu(N)}
  {N^{2s-k}\tau(\ol\chi)}
\sum_{d|b \text{ sat.} \eqref{bD},
\atop{ \gcd(\frac bd,Nd^{(D)})|(r,n)}}
\frac{\psi(nD)\gcd(\frac bd,Nd^{(D)})}{\psi(\frac bd)d^{2s-k}
e^{\frac{2\pi i r\ell }{(b/d)d_D}}}J_\chi(\tfrac bd,d).
\end{align}
Writing $b=ad$, the condition \eqref{bD} becomes \eqref{aD}.
Summing over $t$, we see that
$\frac{e^{2\pi r}}{\nu(N)n^{1-\frac k2}}\sum_{t\in\Q^*}I_{\delta_t}(s)$ is equal to
\[  \hskip -.3cm
 \frac{(4\pi rn)^{k-1}\varphi(D)\psi(nD)e^{i\pi s/2}}{N^sD^{s-k}(k-2)!\,\tau(\ol\chi)}
\hskip -.4cm
\sum_{a\neq 0,d>0\text{ sat.}\eqref{aD},\atop{\gcd(a,Nd^{(D)})|\gcd(r,n)}}
\hskip -.4cm\frac{a^{s-k}\gcd(a,Nd^{(D)})}{d^s\psi(a)e^{\frac{2\pi i r\ell }{ad_D}}}J_\chi(a,d)
{}_1f_1(s;k;\tfrac{-2\pi irnD}{Nad}),
\]
where $a^s=e^{-i\pi s}|a|^s$ if $a<0$.
This completes the proof of Proposition \ref{deltat}.

\section{Asymptotics}\label{Asymp}

Grouping $a$ with $-a$, we rewrite the above sum as follows:
\[\sum_{a,d>0\text{ sat.}\eqref{aD},\atop{\gcd(a,Nd^{(D)})|\gcd(r,n)}}
\hskip -.4cm\left[\frac{a^{s-k}}{\psi(a)e^{\frac{2\pi i r\ell }{ad_D}}}J_\chi(a,d)
{}_1f_1(s;k;\tfrac{-2\pi irnD}{Nad})\right.\]
\[\left.+\frac{e^{-i\pi s}(-1)^ka^{s-k}}{\psi(-1)\psi(a)e^{\frac{-2\pi i r\ell }{ad_D}}}J_\chi(-a,d)
{}_1f_1(s;k;\tfrac{2\pi irnD}{Nad})\right]\frac{\gcd(a,Nd^{(D)})}{d^s}.\]
   From the integral representation \eqref{1f1}, we see that
\begin{equation}\label{B1}
  |{}_1f_1(s,k,2\pi i w)|\le B(\sigma,k-\sigma)\le 1
\end{equation}
 when $1\le \sigma\le k-1$, where $B(x,y)=\int_0^1u^{x-1}(1-u)^{y-1}du
  =\tfrac{\Gamma(x)\Gamma(y)}{\Gamma(x+y)}$ is the Beta function.
  Because $|J_\chi(a,d)|\le 1$, 
  the absolute value of the above is
\[\le \gcd(r,n)B(\sigma,k-\sigma)(1+e^{\pi\tau})\sum_{a,d>0}a^{\sigma-k}d^{-\sigma}
  \qquad(s=\sigma+i\tau).\]
Using $|e^{i\pi s/2}|(1+e^{\pi \tau})
  =2\cosh(\pi\tau/2)$,
  we obtain the following proposition.

\begin{proposition}\label{bound} Write $s=\sigma+i\tau$ for $1<\sigma<k-1$.
  Then the term $E$ given in Proposition \ref{deltat} satisfies the bound
\[
|E|\le \frac{(4\pi rn)^{k-1}D^{k-\sigma-\frac12}\varphi(D)\gcd(r,n)B(\sigma,k-\sigma)}
   {N^\sigma (k-2)!}
  2\cosh(\tfrac{\pi\tau}2)\zeta(k-\sigma)\zeta(\sigma).\]
\end{proposition}

  Theorem \ref{main} now follows immediately.
  In order to prove Corollary \ref{cor}, we must show that the quotient $Q=\frac{E}F$
  has the limit $0$ as $N+k\to\infty$, where 
   $F$ is the first geometric term of \eqref{sum}, and
 $E$ is the error term of \eqref{sum} discussed above.
  We take $k\ge 3$, $N>1$ and $\gcd(n,r)=1$, so for 
  $\frac{k-1}2<\sigma<\frac{k+1}2$, we have 
  $|F|=\frac{2^{k-1}(2\pi rn)^{k-\sigma-1}|\Gamma(s)|}{(k-2)!}$.
  Thus by the above proposition and \eqref{B1},
\begin{equation}\label{Q}
|Q|=|\tfrac{E}F|\ll_{D,\tau} \frac{D^{k-\sigma}(2\pi rn)^\sigma}
  {N^\sigma |\Gamma(s)|}\zeta(k-\sigma)\zeta(\sigma).
\end{equation}
We write $\sigma=\tfrac k2+\delta$ for $|\delta|<\tfrac12$.
  Then each zeta factor is bounded by the constant $\zeta(\tfrac32-|\delta|)$.
By Stirling's approximation (\cite{AS}, 6.1.39),
\[\Gamma(s)^{-1}=\Gamma(\tfrac k2+\delta+i\tau)^{-1}\sim
   \frac{e^{k/2}}{\sqrt{2\pi}(k/2)^{\frac k2+\delta+i\tau-\frac12}}\]
as $k\to\infty$.
With \eqref{Q}, this shows that
\[|Q|\ll \frac{(4D\pi rne)^{k/2}}
  {N^{\frac{k-1}2} k^{\frac k2-1}}
\]
  where the implied constant depends on $\delta,D,r,n,\tau$.
This clearly goes to $0$ as $N+k\to\infty$.

\end{document}